\documentclass[12pt]{article}
\usepackage{graphicx} 
\usepackage[utf8]{inputenc}
\usepackage{graphicx}
\usepackage{amsmath}
\usepackage{amsthm}
\usepackage{authblk}
\usepackage{amssymb}
\usepackage{csquotes}
\usepackage{enumitem}
\newtheorem{theorem}{Theorem}[section]
\newtheorem{definition}{Definition}[section]
\newtheorem{corollary}{Corollary}[section]
\newtheorem{lemma}{Lemma}[section]
\newtheorem{remark}{Remark}
\usepackage[top=3cm, bottom=3cm, right=2.5cm, left=2.5cm,headheight=2.5cm,footskip=15pt, a4paper]{geometry}
\usepackage{newclude}
\usepackage{fancyhdr}
\fancypagestyle{plain}{
\fancyhead[R]{\thepage}
\fancyfoot{}

  \fancyfoot[L]{\small Date: August 6, 2025\\
  Key words and Phrases. Dirac operators, inverse spectral theory, de Branges spaces, canonical systems
  \\
  2020 Mathematics Subject Classification. 34L40 34B40 46C07 81Q10}
}

\usepackage{setspace}
\usepackage{lipsum}
 \usepackage{geometry}
\title{\large \bf GELFAND-LEVITAN CONDITION FOR DIRAC OPERATORS}
\author{\small JIE ZENG }
\date{}
\begin{document}
\maketitle
\begin{abstract}
 We discuss how to generalize a Dirac operator such that the solution of a Dirac equation is of bounded variation rather than continuous. We build the spectral theory for generalized Dirac operators and discuss the connection between them and canonical systems. With the help of de Branges' theory, we discuss the de Branges space of such an operator and the norm endowed. On the other hand, the Paley-Wiener theorem gives us a chance to recover a Dirac operator from a function that plays the same role as the spectral measure, which is well-known as the Gelfand-Levitan condition.
\end{abstract}
\vspace{1cm}
 \textbf{Acknowledgment:} I appreciate Prof. Christian Remling for his support and for providing extremely useful information.
\section{Introduction}
\par
We investigate generalized Dirac operators and some classical Dirac operators on the half line. Even though this question is analogous to that about Sturm-Louisville operators (\cite{17}) and that about one-dimensional Schr\"odinger operators (\cite{16}), there is an essential difference. In Chapter 12 of \cite{3},
they come up with two conditions to characterize the spectral measure of a Dirac operator with a continuous matrix coefficient; however, their method cannot be used directly if the coefficient is not continuous. In this paper, we take advantage of de Branges theory (\cite{4},\cite{5},\cite{6},\cite{7}, and \cite{8} for more compatible details) to deduce a G-L function (see Section 5) from a Dirac operator, and reconstruct explicitly to some degree a Dirac operator from a G-L function. In this paper, we also rely on the relation between Dirac operators and canonical systems, which will be introduced in Section 2.
\par 
A Dirac equation (or system) is a differential equation of the form
\begin{equation}
   Jf'-\mu f=g 
\end{equation}
where $\mu=\begin{pmatrix} 
					\mu_1&\mu_2\\
					\mu_2&-\mu_1 \\
\end{pmatrix}
$ is a $2\times2$ measure on the Borel sets of $[0,\infty)$ satisfying $|\mu_i|([0,N])< \infty$ for all $N>0$, and $J=\begin{pmatrix} 
					0 & -1\\
					1 & 0 \\
\end{pmatrix}$.
\par
Suppose the measure $\mu$ is absolutely continuous with respect to the Lebesgue measure. In that case, we get a classical Dirac differential equation, and it is easy and natural to define an operator in convention, see \cite{9} for instance. If $\mu$ is a \textquotedblleft real\textquotedblright measure, we first have to ask about the existence and uniqueness of the solution of such a differential equation, which should be trivial in the classical scenario. One straightforward counterexample may be $\mu= \begin{pmatrix} 
					0 &\delta \\
					\delta &0 \\
\end{pmatrix}$ where \(\delta\) is Dirac measure at some point \(x_0\) if we follow the convention. We will introduce a compatible interpretation to explain this differential equation as an integral equation, and define a Dirac operator in Section 2.
\par
A canonical system is defined as follows:
\begin{equation}
u'(x)=zJH(x)u(x)
\end{equation}
on an open interval $x\in(a,b)$, $-\infty\leq a< b\leq\infty$, where $z$ is a complex number, and $H$ satisfies: (1) $H\in \mathbb{R}^{2\times 2}$, \ (2) $H\in L_{loc}^1(a,b)$, \ (3) $H$ is Hermitian and non-negative definite for (Lebesgue) almost all $x\in (a,b)$.
\par
The proper Hilbert space when talking about a canonical system is not $L^2(a,b)$ anymore, but a space called $L^2_H(a,b)$ instead:
\par
Let's define 
\[\mathcal{L}=\{f:(a,b)\to \mathbb{C}^2: f \textit{(Borel) measurable,}\ \int_a^bf^*Hf<\infty\}\]
with (pre)norm $||f||=(\int_a^bf^*Hf)^\frac{1}{2}$.
\par
The Hilbert space is defined as the quotient
\[L^2_H(a,b):=\mathcal{L}/\mathcal{N}\]
where $\mathcal{N}=\{f\in \mathcal{L}:||f||=0\}$.
\par
A canonical system has numerous elegant conclusions, and the most relevant one here is that there is a bijection between canonical systems and (generalized) Dirac systems. With the help of this connection, it is possible to characterize a nice Dirac operator by an $L^1$ function. 
\par
This paper is organized in the following way:
\begin{itemize}
    \item In Section 2, we introduce generalized Dirac operators and go over some basic properties. The connection between canonical systems and (generalized) Dirac systems can be found there as well. More details can be found in \cite{12}.
    \item In Section 3, we introduce de Branges theory and use this theory to show that the spectral measure of a classical Dirac operator on an arbitrary interval $[0,N]$ can be realized by a function $f\in \Phi_N$ (see Definition \ref{d3.1}); more generally, even the operator is not classical, we may still generate its de Branges space as a Paley-Wiener space, this is our \textbf{main result (Thoerem \ref{3.4})}. To take advantage of this result, one may consider a half-line problem by taking $N\to \infty$, and weak convergence of spectral measures (uniformly compact convergence of Weyl functions) will give us an analogous conclusion.
    \item In Section 4, we reconstruct a Dirac operator on $[0,N]$ if a function $f\in \Phi_N$ is given.  Another \textbf{main result (Theorem \ref{t4.3})} emphasizes this recovery in the form of canonical systems. Once again, it is not hard to generalize this result to a half-line situation. Moreover, we also consider a special situation $f\in \Phi_N\cap L^2$, and we \textbf{conclude (Theorem \ref{t4.4})} that in this case recovered Dirac operator is classical, namely, the measure $\mu$ in the equation (1) is replaced by an $L^1$ function. Based on the theorems mentioned above, we discuss the inverse spectral theory in Gelfand-Levitan style (\textbf{Theorem \ref{t5.1}}). 
\end{itemize}

\vspace{2cm}
\section{Dirac Operators and Canonical Systems}
\par In this section, we review some fundamental material on Dirac systems and canonical systems. Most topics about canonical systems can be found in \cite{1},  Dirac systems in \cite{12}, and another modification for Dirac systems in Section 2 of \cite{10}.
\par
Let N be a positive number. We call a function $f \in L^1[0,N]$ from $\mathbb{R}$ to $\mathbb{C}$ of bounded variation (on $[0,N]$), if the total variation of $f$, defined by 
\[V_0^N(f):=\sup\{\int_0^N|f(t)|\phi'(t)dt:\phi\in C_c^1([0,N],\mathbb{R}),||\phi||_{L^\infty }\leq1\}\]
is finite. And we use the notation $BV[0,N]$ to represent the collection of all bounded variation functions on $[0,N]$, i.e.,
\[BV[0,N]:=\{f \in L^1[0,N]:V_0^N(f)<\infty\}\]
This definition is equivalent to that the real part and the imaginary part of $f$ are of bounded variation in the sense of real functions.
\par
If $f=\begin{pmatrix}
f_1\\
f_2\\
\end{pmatrix}$ is a function from $\mathbb{R}$ to $\mathbb{C}\times\mathbb{C}$, we call $f$ is of bounded variation (on $[0,N]$) if $f_1$ and $f_2$ are in $BV[0,N]$, and we also use the same notation $BV[0,N]$, i.e., 
\[BV[0,N]:=\{f=\begin{pmatrix}
f_1\\
f_2\\
\end{pmatrix} \in L^1[0,N]:V_0^N(f_i)<\infty ,i=1,2\}\]
We also define the total variation of $f$ by
\[V_0^N(f):=\max_{i=1,2}(V_0^N(f_i))\]
When considering the half-line problem, it is natural to define the space of all functions of locally bounded variation by 
\[BV[0,\infty):=\{f=\begin{pmatrix}
f_1\\
f_2\\
\end{pmatrix} \in L^1_{loc}[0,\infty):V_0^N(f)<\infty\ \textit{ for all}\ N>0\}\]
Analogously, we say a matrix is in $BV[0,N(\infty)]$ if all entries are in $BV[0,N(\infty)]$. This definition can be alternated by the classical one as the summation of differences when necessary (Chapter 7 in \cite{11} by Giovanni Leoni).
\par
Let $\mu$ be a $2\times2$ signed Borel measure on $[0,\infty)$ of the form $\begin{pmatrix} 
					\mu_1 & \mu_2\\
					\mu_2 & -\mu_1 \\
\end{pmatrix}$, we define the set of such measures as follows:
\[DS:=\{\mu:(1) \max_{i=1,2}(|\mu_i|([0,N])<\infty\ \textit{for all}\ N>0;\ (2)\ \mu(\{0\})=0\}\]
Condition (2) is just a normalization, and there is no technical difficulty coming to remove it with the help of Lemma \ref{2.1}; however, condition (1) is essential when we apply Jan Persson's results (\cite{2}); moreover, this condition also implies that there are only countably many mass points in any compact subset of $[0,\infty)$, namely, $\mu\{x\}\neq 0$ at those points. We cavalierly shorten our notations as  $\mu\{x\}$ and $ \mu(a,b)$ for their merited meanings and say \enquote{a Dirac stuff $\mu$ in $DS$}.
\par
Assume $\mu\in DS$ and $f\in L^1_{loc}(\mu)$,we interpret integral $\int_{a}^{x}d\mu f$ as follows:
\begin{equation*}
\int_{a}^{x}d\mu f=\left\{\begin{array}{rcl}
\int_{(a,x]}d\mu f&&x\geq a\\
-\int_{(x,a]}d\mu f&&x<a
\end{array}\right.
\end{equation*}
Recall that if $f,h$ are functions of locally bounded variation from $\mathbb{R}$ to $\mathbb{C}$, then integration by parts is given by
\[\int_{[a,b]}f(x+)dh+\int_{[a,b]}dfh(x-)=f(b+)h(b+)-f(a-)h(a-)\]
Here, $d$ means the relevant Lebesgue–Stieltjes measure associated with the right-continuous representation of the function, see formula (21.68) in \cite{13} for example.
\par
To define a generalized Dirac operator, we introduce a function, denoted by $g$, from $\mathbb{C}^{2\times2}$ to $\mathbb{C}^{2\times2}$:
\[g(D):=\sum_{n=1}^{\infty}\frac{D^{n-1}}{n!}\]
\par
Let $\mu\in DS$, define \[Af(x):=Jf(x)-\int_{0}^{x}g(\mu \{s\}J)d\mu (s) f(s)\]
for $f\in BV[0,\infty)$ as its right continuous representation.
\begin{definition}
\label{Dirac operators}
 For $\mu\in DS$, the following operator $T$ is called a (generalized) Dirac operator on $L^2[0,\infty)$:
 \[D(T):=\{f\in L^2[0,\infty):f\in BV[0,\infty)\ \textit{and right continuous,}\]
 \[Af \in AC[0,\infty), (Af)'\in L^2[0,\infty) \}\]
\[Tf:=-(Af)'\]
 \end{definition}
 \begin{remark}
 As expected, we say $f\in AC[0,\infty)$ if $f\in AC[0,b]$ for all $b>0$, i.e, $f$ is locally absolutely continuous. Also, if the measure $\mu$ is continuous, then $g(\mu \{s\}J)$ is the identity, so we go back to the conventional definition.  
 \end{remark}
 \par
 Let's take a closer look at this definition. let $k\in L^2[0,\infty)$, then the equation $Tf=k$ is equivalent to 
 \begin{equation}
 Jf(x)-\int_{0}^{x}g(\mu \{s\}J)d\mu f=C-\int_0^xkdt
\end{equation}
The existence and uniqueness of the solution of this integral equation are the keys to making $T$ an operator. The following result belongs to Jan Persson, but we have modified his results for convenience:
\begin{theorem}[\cite{2}]
 Let $A$ be a $2\times2$ complex Borel measure on the real line and $I$ the identity matrix. Let $k$ be a $2\times 1$ complex Borel measure on the real line. If $A\{x\}+I$ is invertible for all $x\in\mathbb{R}$, then to each choice of $ C\in\mathbb{C}^2$, there is a unique solution $f$, which is of locally bounded variation and right continuous, of 
\[f(x)=C-\int_0^xdA(t)f+\int_0^xdk(t),\ x\geq 0\]
and
\[f(x)=C+\int_0^xdA(t)f-\int_0^xdk(t),\ x< 0\]   
\end{theorem}
\par
Notice that if we pick up $dA=Jg(\mu \{s\}J)d\mu$, then  $A\{x\}+I=e^{J\mu \{s\}}$, which is invertible, and thus we know Definition \ref{Dirac operators} makes sense.
\par
An analogous one, $T_N$ on the interval $[0,N]$, can be defined in the same way:
\begin{definition}
 \[D(T_N):=\{f\in L^2[0,N]:f\in BV[0,N]\ \textit{and right continuous,} \]
\[Af \in AC[0,N], (Af)'\in L^2[0,N] \}\]
\[T_Nf:=-(Af)' \]   
\end{definition}
\begin{lemma}
\label{2.1}
Let $f\in D(T_N)(D(T))$, then $f(x-)=e^{J\mu\{x\}}f(x)$. As consequences, $f$ is continuous at $\mu\{x\}=0$ and $f^*(x-)=f^*(x)e^{-\mu\{x\}J}$.    
\end{lemma}
\begin{proof}
 As $Af \in AC[0,N]$, we have $\lim\limits_{y \to x^-}Af(y)=Af(x)$, then it follows that $Jf(x-)=Jf(x)+g(\mu \{x\}J)\mu\{x\}f(x)=Je^{J\mu\{x\}}f(x)$.   
\end{proof}
\par
We introduce a useful tool called the transfer matrix of the operator $T_N(T)$. Basically speaking, the transfer matrix $T(x)$ is a $2\times 2$ matrix-valued solution of \[JT(x)-\int_{0}^{x}g(\mu \{s\}J)d\mu T=I\]
where $I$ is the identity.
\par
Obviously, $T(0)=I$. we sometimes need to write  down $T$ explicitly as $T(x)=(u(x),v(x))=\begin{pmatrix} 
					u_1 & v_1\\
					u_2 & v_2 \\
\end{pmatrix}$, where $u$ and $v$ are vector-valued solutions of the integral equation satisfying $u(0)=\begin{pmatrix} 
					1\\
					0 \\
\end{pmatrix}$ and $v(0)=\begin{pmatrix} 
					0\\
					1\\
\end{pmatrix}$ respectively. Of course, $u$ and $v$ are right continuous and of locally bounded variation.
\par
The operators have all the desired properties, for instance,
\begin{lemma}
$det(T(x))=1$ for $x\in[0,\infty)$.    
\end{lemma}
\begin{lemma}[Variation of Constants]
 Let $f\in D(T_N)(D(T))$ and $T_Nf=k(Tf=k)$. Assume $f(0)=C$, then
\begin{equation}
    f(x)=T(x)C+T(x)\int_0^xT^{-1}Jkdt
\end{equation}   
\end{lemma}
The proofs are not trivial and rely on the other two results in \cite{2} due to J.Persson and an observation that $T^{-1}(x)f(x)$ is continuous due to Lemma \ref{2.1}, we don't want to repeat those two proofs here, for reader who are interested in them, please see Claim 3.2, Claim 3.3 in \cite{12}.
\par
The next conclusion states that $T$ and $T_N$ are densely defined.
\begin{lemma}
 $\overline{D(T_N)}=L^2[0,N]$, $\overline{D(T)}=L^2[0,\infty)$.   
\end{lemma}
\begin{proof}
Let $f\in L^2[0,N]$, we have $T^{-1}f\in L^2[0,N]$ as $T^{-1}$ is bounded on $[0,N]$ under the supremum norm. We can pick up a sequence $\{C_n\}_{n=1}^\infty\subset C_0^\infty(0,N)$ such that $C_n\stackrel{L^2}{\longrightarrow}T^{-1}f$, then it follows that $TC_n\stackrel{L^2}{\longrightarrow}f$.
\par
We define $k_n(x):=-JT(x)C_n'(x)$, then $k_n\in L^2[0,N]$ since $C_n'$ is also bounded. Moreover, $C_n(x)=\int_0^xT^{-1}Jk_ndt$. Now, we consider \[Jf_n(x)-\int_{0}^{x}g(\mu \{s\}J)d\mu f_n=-\int_0^xk_ndt\]
Theorem \ref{2.1} shows that $f_n\in D(T_N)$, and (4) shows \[f_n(x)=T(x)\int_0^xT^{-1}Jk_ndt\] i.e., $f_n(x)=T(x)C_n(x)$.
It follows from $f_n\stackrel{L^2}{\longrightarrow}f$ that $\overline{D(T_N)}=L^2[0,N]$.
\par
To show $\overline{D(T)}=L^2[0,\infty)$, we pick up $f\in L^2[0,\infty)$. Since $\chi_{[0,N]}f\stackrel{L^2}{\longrightarrow}f$, we can pick up $\{C_n\}_{n=1}^\infty\subset C_0^\infty(0,N)$ such that $f_n=TC_n\stackrel{L^2}{\longrightarrow}\chi_{[0,N]}f$. Observe that $f_n(N-)=0$, then Lemma \ref{2.1} shows $f_n(N)=0$, hence the function \[\widetilde{f_n}=\left\{\begin{array}{rcl}
f_n(x)&&{x\leq N}\\
0&&{x>N}
\end{array}\right.\] is in $D(T)$. it follows from this approximation that $\overline{D(T)}=L^2[0,\infty)$.
\end{proof}
As a densely defined operator on Hilbert space, it is extremely natural to discuss its spectral theory. This is a tedious but standard process, so we will just briefly state the scheme. For readers who want more details, please see Chapter 4 in \cite{12}, and \cite{1} for a more general treatment.
\par
One can define the Wronskian of two functions $f$ and $h$ by $W_{f,h}(x):=(f^*Jh)(x)$, and this is used in the following famous equation:
\[\langle T_{(N)}f,h\rangle-\langle f,T_{(N)}h\rangle=\lim\limits_{x \to \infty(N)}W_{f,h}(x)-W_{f,h}(0)\]
as the inner product in $L^2$.
By studying the adjoint of $T_N$ and thanks to von Neumann Theory of symmetric relations (see \cite{1} for instance), we can characterize all self-adjoint realizations of $T_N$, specifically, two boundary conditions can be given at 0 and $N$ respectively to get a self-adjoint operator. 
\par
we consider the self-adjoint realization of $T_N$ with separated boundary condition: $f_2(0)=0,e_{\beta}^*Jf(N)=0$ for $e_{\beta}=\begin{pmatrix}
\cos{\beta}\\
\sin{\beta}\\
\end{pmatrix}$ with some $\beta \in [0,\pi)$. The eigenvalue problem $T_Nf=zf$ will give the Weyl-Titchmarsh function 
\begin{equation*}
 m_N^\beta(z):=\frac{f_1(0,z)}{f_2(0,z)}=T^{-1}(N,z)\cot{\beta} 
\end{equation*}
where $f$ is the nontrivial solution of $T_Nf=zf$ satisfying the boundary condition at $N$, $T(x,z)$ is a $2\times 2$ solution of 
\[JT(x,z)-\int_{0}^{x}g(\mu \{s\}J)d\mu T(t,z)=J-z\int_0^xT(t,z)dt\]
And the second equation should be interpreted as the Mobius transform. $T(x,z)$ is the generalization of the transfer matrix and is entire for any fixed $x$ with the help of Arzela-Ascoli theorem. By applying standard Weyl theory, we conclude that if we assume the limit point case at $\infty$ (we will see later that this is the only situation), i.e., there is a unique non-trivial solution $f$ up to a factor (in $BV[0,\infty)$ and right continuous) of $(Af)'=-zf$ which is in $L^2[0,\infty)$ for all $z\in \mathbb{C}^+$ (there is another way to define it, but by the argument of deficiency indices, there is no essential difference), then the Weyl-Titchmarsh function for the half line can be given by
\[m(z)=\frac{f_1(0,z)}{f_2(0,z)}\]
\par
Now, let's turn to canonical systems (2), and we recommend \cite{1} to readers for the vicissitude of this topic.
\par
We denote all canonical systems by $H\in CS$ and a subset of it by
\[CD:=\{H\in CS:(1) H\in BV[0,\infty)\ \textit{and right continuous;}\]\[ (2)detH=1;\ (3)H(0)=1\}\]
If we recall a super elegant result that $\infty$ is in the limit point case if and only if $trH\notin L^1$ near $\infty$ [Theorem 3.5, \cite{1}], we conclude that for $H\in CD$, it is in the limit point case since it follows from $detH=1$ and $H\geq 0$ that $trH\geq 2$. This observation shows that the Weyl-Titchmarsh function of $H$, denoted by $m(z,H)$, can be defined. In Chapter 5 of \cite{12}, we have proved that for such a canonical system, there is $\mu \in DS$ with $m(z,H)$ as its Weyl-Titchmarsh function, and vice versa. In summary, we have the following theorem:
\begin{theorem}
\label{2.2}
 There is a bijection between $DS$ and $CD$ in the sense of sharing the same Weyl-Titchmarsh function.   
\end{theorem}
Moreover, we can say more about this theorem:
\begin{corollary}
 \label{cor2.1}
 Under Theorem \ref{2.2}, the bijection maps the subset of $DS$ with an extra condition that $\mu$ is absolutely continuous w.r.t. the Lebesgue measure on the subset of $CD$ with an extra condition that $H\in AC[0,\infty)$.   \end{corollary}
 \begin{corollary}
 \label{c2.2}
Under Theorem \ref{2.2}, the bijection maps the subset of $DS$ with an extra condition that $\mu$ is continuous w.r.t. the Lebesgue measure on the subset of $CD$ with an extra condition that $H\in C[0,\infty)$.     
 \end{corollary}
 \par
 One can see \cite{9} to find a straightforward calculation for Corollary \ref{cor2.1}, even though the calculation doesn't work for a real measure rather than an $L^1_{loc}$ function; moreover, as we have discussed before, an analogous conclusion as Corollary 2.2 in \cite{9} can be proved:
 \begin{corollary}
 Operator $T$ is in the limit point case at $\infty$.    
 \end{corollary}
 We leave the proof to readers.
 Now, we have all the ingredients to set up the spectral representation theorem:
 \begin{theorem}
 \label{2.3}
  \[Uf=\int_0^\infty u^*(s,t)f(s)ds,f\in \underset{N>0}{\cup}L^2[0,N]\]
\[Uf=\lim\limits_{N \to \infty}U(\chi_{[0,N]}f),f\in L^2[0,\infty)\]
define a unitary map $U: L^2[0,\infty)\longrightarrow L^2(\mathbb{R},\rho)$ (here, limit is norm limit in $L^2(\mathbb{R},\rho)$).\\
Then this map, together with the spectral measure $\rho$ from the Weyl-Titchmarsh function, provides a spectral representation.   
 \end{theorem}
 This is directly from Theorem 3.17 of \cite{1} once we apply Theorem \ref{2.2}.

\vspace{2cm}
\section{De Branges Spaces of Dirac
Operators}
\par
In this section, we discuss de Branges theory. In de Branges's brilliant papers \cite{4},\cite{5},\cite{6},\cite{7}, he came up with a lot of tools we need to use in this paper; however, we refer to \cite{8} since it has a modified version we will take advantage of. A canonical-system version of the theory can be found in \cite{1}, and more details about the Dirac version are in \cite{12}.
\subsection{Fundamental Properties of De Brange Spaces}
\par
A de Branges (dB) function is an entire function $E$ such that $|E(z)|>|E^{\#}(z)|$ for $z\in \mathbb{C}^+$. Here, $E^{\#}(z)=\overline{E(\overline{z})}$. The de Branges space of $E$ is defined as 
\[B(E):=\{F: F\ \textit{entire},\ \frac{F}{E},\frac{F^{\#}}{E}\in H^2\}\]
where $H^2=H^2(\mathbb{C}^+)$ is the Hardy space on the upper half plane.
\par
Endowed with the following inner product
\[[F,G]=\frac{1}{\pi}\int_\mathbb{R}\overline{F}(t)G(t)\frac{dt}{|E(t)|^2}\]
$B(E)$ becomes a Hilbert space. Moreover, the reproducing kernels 
\[J_w(z)=\frac{\overline{E}(w)E(z)-\overline{E^{\#}}(w)E^{\#}(z)}{2i(\overline{w}-z)}\]
are in $B(E)$, and $[J_w,F]=F(w)$ for all $F\in B(E),w\in \mathbb{C}$.
\par
A dB space is determined uniquely by a dB function up to a constant, namely, $B(E_1)=B(E_2)$ if and only if 
\[\begin{pmatrix}
ReE_2\\
ImE_2\\
\end{pmatrix}=M\begin{pmatrix}
ReE_1\\
ImE_1\\
\end{pmatrix}\]
for some $M\in SL(2,\mathbb{R})$. Here, $B(E_1)=B(E_2)$ means they share the same functions and are isometrically equal to each other as Hilbert spaces.
\par
Moreover, We call a dB space $B(E)$ regular, if for all $z_0\in \mathbb{C}$,
\[F\in B(E)\Rightarrow S_{z_0}F(z):=\frac{F(z)-F(z_0)}{z-z_0}\in B(E)\]
where $S_{z_0}F(z_0)=\lim\limits_{z\to z_0}S_{z_0}F(z)$.
\par
On the other hand, the following theorem characterizes a Hilbert space as a dB space, which belongs to de Branges:
\begin{theorem}
\label{3.1}
Let $H$ be a Hilbert space whose elements are entire functions. Assume that:\\
1) For every $z\in \mathbb{C}$, point evaluation $z(F):=F(z)\in H^*$;\\
2) If $F\in H$ with $F(w)=0$, then $G(z)=\frac{z-\overline{w}}{z-w}F(z)\in H$ and $||F||=||G||$;\\
3) $F\to F^{\#}$ is isometric on $H$.\\
Then $H=B(E)$ for some de Branges function $E$.
\\
Conversely, if $B(E)$ is a de Branges space, then it satisfies those assumptions above.    
\end{theorem}
The following de Branges' theorem will play a significant role later:
\begin{theorem}[the Ordering Theorem]
Let $B(E),B(E_1),B(E_2)$ be regular dB spaces and $B(E_1),B(E_2)\subset B(E)$, then either $B(E_1)\subset B(E_2)$ or $B(E_2)\subset B(E_1)$.    
\end{theorem}
There is a profound connection between dB spaces and canonical systems due to de Branges, which is the key point to construct our results:
\begin{theorem}
If $B(E)$ is a regular dB space, $E(0)=1$ and $N>0$, then there is a coefficient $H(x)\in L^1(0,N)$ of some canonical system on $(0,N)$ such that $E(z)=u_1(N,z)-iu_2(N,z)$. Moreover, $H$ can be chosen so that $trH$ is a positive constant. Here, $u$ is the solution of the equation (2) with the initial value  
$u(0,z)=\begin{pmatrix}
    1\\
    0\\
\end{pmatrix}$.   
\end{theorem}
\par
Now, we turn to the type of an entire function. An entire function $F$ is said to be of exponential type if $|F(z)|\leq C(\tau)e^{\tau|z|}, (z\in \mathbb{C})$ for some $\tau>0$. the infimum of the $\tau>0$, denoted by $\tau (F)$, is called the type of $F$. 
\par
If we consider a canonical system on $[0,N]$ with $H\in L^1(0,N)$, then the type of $E_N(z)=u_1(N,z)-iu_2(N,z)$ is given by (Theorem 4.26,\cite{1})
\[\tau_N (E)=\int_0^N\sqrt{detH(x)}dx\]
Another significant conclusion related to the type from a canonical system is that The space $B_\tau:=\{F\in B(E_N):\tau (F)\leq \tau\}$ is also a de Branges space and $B_\tau=B(E_a)$, where $a=\max\{x\in \mathbb{R}:\tau_x(E)\leq \tau\}$.
\subsection{dB Spaces of Dirac Operators}
 Let $N<\infty$, $\mu\in DS$ and we consider $T_N$. An easy calculation  by the definition shows that $E_N(z)=u_1(N,z)-iu_2(N,z)$ is a dB function, and the reproducing kernels of $B(E_N)$ are given by
\[J_w(z)=\int_0^Nu^*(x,w)u(x,z)dx\]
Recall Theorem \ref{2.3} that the unitary map $U$ is related to the spectral measure; on the other hand, $U$ is also related to a dB space as well, precisely, 
\[Uf(z)=\int_0^Nu^*(x,\overline{z})f(x)dx\]
defines an isometry $U:L^2[0,N]\to B(E_N)$. This conclusion is easy to see if we notice its canonical-system version (Theorem 4.7,\cite{1}) and that a Dirac system is a special canonical system. Another observation from this isometry is that if $N_1<N_2$, then $B(E_{N_1})\subset B(E_{N_2})$. Moreover, if $F\in B(E_{N_1})$ and $|h|\leq \tau_{N_2}-\tau_{N_1}$ where $\tau_{N_i}$ is the type of $E_{N_i}$, then $e^{ihz}F\in B(E_{N_2})$.
\par
Now, we can state our first result:
\begin{theorem}
\label{3.4}
Assume $\mu\in DS$. Then for any arbitrary $N \in (0,\infty)$, $B(E_N)$ is a Paley-Wiener space as sets, i.e., 
\[B(E_N)=PW_N:=\{F(z)=\widehat{f}(z): f\in L^2(-N,N)\}\]
Moreover, if $d\mu\ll dt$ w.r.t. Lebesgue measure, then there is $\phi\in L^1(\mathbb{R})$ satisfying $\phi(x)=\overline{\phi}(-x)$ such that $\frac{1}{|E_N(t)|^2}=1+\widehat{\phi}(t)$ on $\mathbb{R}$, where $\widehat{\rho}(z)=\int_\mathbb{R}e^{izs}d\rho(s)$ is the Fourier transform. In this case, the norm of $B(E_N)$ is given as follows for $F=\widehat{f}\in B(E_N)$: 
\[\lVert F\rVert^2_{B(E_N)}=2\langle f,f+\phi*f\rangle\]
\end{theorem}
Since $\lVert F\rVert^2_{B(E_N)}$ is determined by $\phi*f$ on $(-N,N)$, that is, by the restriction of $\phi$ to $(-2N,2N)$, hence without loss of generality, we can assume $\phi\in L^1(-2N,2N)$; moreover, as a norm, we require that $\langle f,f+\phi*f\rangle>0$ for non-zero $f\in L^2(-N,N)$. Those two conditions are used to define a subset of $L^1(\mathbb{R})$,denoted by $\Phi_N$, as follows:
\begin{definition}
\label{d3.1}
\[\Phi_N:=\{\phi\in L^1(-2N,2N):1) \phi(x)=\overline{\phi}(-x); 2)\langle f,f+\phi*f\rangle>0, 0\neq f\in L^2(-N,N)\}\]    
\end{definition}
Roughly speaking, Theorem \ref{3.4} says that if we have a classical Dirac operator with an $L^1_{loc}$ coefficient, then the norm of $B(E_N)$ is determined by a function in $\Phi_N$.
\subsection{the Proof of Theorem \ref{3.4}}
Consider $[0,N]$ and pick $\mu \in DS$. We define the set of all jump points of $\mu$ as follows:
\[S(\mu):=\{x\in(0,\infty): \mu\{x\}\neq 0\}\] 
and a subset:
\[S_N(\mu):=\{x\in[0,N]:x\in S(\mu)\}=\{x_1, x_2, x_3,\cdots\}\]
and a function on $(0,\infty)$ by
\[t(x):=(\mu_1^2\{x\}+\mu_2^2\{x\})^\frac{1}{2}\]
Obviously, $t(x)\neq 0$ if and only if $x\in S(\mu)$. We introduce a measure $\omega$ which is defined as follows:
\[\omega:=\sum\limits_{x\in S_N(\mu)}\frac{e^{t(x)}+e^{-t(x)}-2}{2}\delta_x\]
where $\delta_x$ is the Dirac measure at $x$.
\begin{lemma}
$\omega$ is a complex measure.    
\end{lemma}
\begin{proof}
 The total variation of $\omega$ is given by
 \[|\omega|(\mathbb{R})=\frac{1}{2}\sum\limits_{x\in S_N(\mu)}(e^{t(x)}+e^{-t(x)}-2)=\frac{1}{2}\sum\limits_{x\in S_N(\mu)}(\sum_{n=1}^\infty\frac{t^{2n}(x)}{(2n)!})\]
 Notice that 
 \[\frac{1}{2}\sum_{n=1}^\infty\frac{\sum\limits_{x\in S_N(\mu)}t^{2n}(x)}{(2n)!}\leq \frac{1}{2}\sum_{n=1}^\infty\frac{(\sum\limits_{x\in S_N(\mu)}t(x))^{2n}}{(2n)!}=\frac{1}{2}(e^{\sum t(x)}+e^{-\sum t(x)}-2)<\infty\]
 The last inequality comes from $\sum\limits_{x\in S_N(\mu)}t(x)<\infty$ as $\mu$ is locally complex.
\end{proof}
\begin{corollary}
There is a unique function $k\in BV[0,\infty)$ which is right continuous such that 
\[k(x)=1-\int_0^xd\omega k\]
As a consequence, the collection of all jump points of $k$ is exactly $S_N(\mu)$.    
\end{corollary}
\begin{proof}
we have for $x\in S_N(\mu)$,
\[\omega\{x\}+1=\frac{e^{t(x)}+e^{-t(x)}}{2}\neq 0\]
It follows from Jan Persson's theorem (Theorem \ref{2.1}) with a diagonal $2\times 2$ matrix with entries $k$. One can also invoke \cite{2} for the one-dimensional case.
\end{proof}
The following lemma is essential.
\begin{lemma}
There are $\epsilon,M>0$ such that $\epsilon\leq |k|\leq M$.    
\end{lemma}
\begin{proof}
 Since $k\in BV[0,\infty)$ is right continuous, the existence of $M$ is trivial. On the other hand, we claim that $k(x)\neq 0$. Indeed, if $k(x_0)=0$ for some $x_0$, then we have $0=1-\int_0^{x_0}d\omega k$, hence the equation
\[k(x)=(1-\int_0^{x_0}d\omega k)-\int_{x_0}^xd\omega k=-\int_{x_0}^xd\omega k, x\geq x_0\]
\[k(x)=\int^{x_0}_xd\omega k, x<x_0\]
has a unique solution which is still $k(x)$; however, 0 is the solution of this equation, which means $k=0$. This contradicts with $k(0)=1$.
\par
Recall that $k(x)=\frac{k(x-)}{1+\omega\{x\}}$; moreover, we also have that $k(x)=1$ for $x<0$ and $k(x)=k(N)$ for $x>N$. If $\inf\limits_{x\in [0,N]}|k|=0$, then there is a sequence $\{x_n\}_{n=1}^\infty \subset[0,N]$ such that for any $n\in \mathbb{N}^+$,
\[|k(x_n)|<\frac{1}{n}\]
Since $[0,N]$ is compact, there is a convergent subsequence of $\{x_n\}_{n=1}^\infty$, still denoted by $\{x_n\}_{n=1}^\infty$. In other words, $\lim\limits_{n\to \infty}x_n=x_0\in [0,N]$.
\par
If $x_n$ approximates to $x_0$ from the right-hand side, i.e., there is a subsequence $\{x_{n_k}\}_{k=1}^\infty$ such that $x_0<x_{n_k}$ and $\lim\limits_{k\to \infty}x_{n_k}=x_0$, then by the existence of the right limit of a function of bounded variation, it follows that 
\[|k(x_0)|=|k(x_0+)|=\lim\limits_{k\to \infty}|k(x_{n_k})|\leq \lim\limits_{k\to \infty}\frac{1}{n_k}=0\]
If $x_n$ can approximate to $x_0$ from the left-hand side, then we still have such a subsequence, hence 
\[|k(x_0)|=\frac{|k(x_0-)|}{1+\omega\{x_0\}}=\frac{\lim\limits_{k\to \infty}|k(x_{n_k})|}{1+\omega\{x_0\}}\leq \lim\limits_{k\to \infty}\frac{1}{n_k(1+\omega\{x_0\})}=0\]
In all, we conclude that $k(x_0)=0$, which is a contradiction. Since $\inf\limits_{x\in [0,N]}|k|>0$, we simply say $\epsilon=\inf\limits_{x\in [0,N]}|k|$.
\end{proof}
Assume $y$ is the solution of the equation 
\[Jy(x)-\int_{0}^{x}g(\mu \{s\}J)d\mu y=C-z\int_0^xydt\]
We define a function
\[f:=\frac{1}{k}Q^{-1}y\]
where $Q(x,z)=\begin{pmatrix} 
e^{izx} & e^{-izx}\\
-ie^{izx} & ie^{-izx} \\
\end{pmatrix}$.
\begin{lemma}
\label{3.3}
On the interval $[0,N]$, $f$ satisfies 
\begin{equation}
    f(x)=\frac{1}{2}\begin{pmatrix} 
-i & 1\\
i & 1\\
\end{pmatrix}C+\int_0^x\begin{pmatrix} 
					0 & e^{-2izs}dP(s)\\
					e^{2izs}d\overline{P}(s) & 0\\
\end{pmatrix}f(s)
\end{equation}
where $dP(s)=\frac{e^{t(s)}-e^{-t(s)}}{t(s)(e^{t(s)}+e^{-t(s)})}(d\mu_2(s)-id\mu_1(s))$.    
\end{lemma}
To prove this lemma, we introduce a generalized version of the chain rule due to Volpert. For our purpose, we only need the following weaker version, and for readers who are interested in a global version, please see \cite{14}.
\begin{theorem}[Volpert's chain rule]
  Let $I\subset \mathbb{R}$ be an open interval, $f:\mathbb{R}^m\rightarrow\mathbb{R}^n$ continuously differentiable, $u: I\to \mathbb{R}^m$ is of bounded variation, and $S$ the set of all jump points of $u$ defined as the set of all $x\in I$ where the approximate
limit $\Tilde{u}$ does not exist at $x$. Then
\[d(f(u))=du\cdot df(\Tilde{u})\]
in the sense of measures on $I\setminus S$, where $d$ is the distributional derivative.  
\end{theorem}
Now we can prove Lemma \ref{3.3}.
\begin{proof}
First of all, we claim that 
\begin{equation}
    Jg(\mu\{x\}J)=\frac{e^{t(x)}-e^{-t(x)}}{2t(x)}J+\frac{e^{t(x)}+e^{-t(x)}-2}{2t^2(x)}\mu\{x\}
    \end{equation}
here, coefficients should be interpreted as limits, i.e., when $t(x)=0$, then $\frac{e^{t(x)}-e^{-t(x)}}{2t(x)}=1$ and $\frac{e^{t(x)}+e^{-t(x)}-2}{2t^2(x)}=\frac{1}{2}$.
\par
Indeed, if $t(x)=0$, then $\mu\{x\}=0$, which implies $Jg(0)=J$. If $t(x)\neq0$, and notice that $(\mu\{x\}J)^2=t^2(x)I$, we have
\[g(\mu J)=\sum\limits_{n=1}^\infty\frac{(\mu J)^{n-1}}{n!}\\
=\frac{e^t-e^{-t}}{2t}I+\frac{e^t+e^{-t}-2}{2t^2}\mu J\]
Observe that $J\mu J=\mu$, thus we get the desired identity.
\par
Now, we fix $z\in \mathbb{C}$. Since $f$ is of bounded variation on $[0,N]$ and right continuous, and $y(0-)=y(0)$ as $\mu\{0\}=0$, we have
\begin{equation}
\int_{[0,x]}d(kQ)f=k(x)Q(x)f(x)-y(0)-\int_{[0,x]}k(s-)Q(s)df(s)    
\end{equation}
by integration by parts on $[0,N]$.
\par
By Volpert's chain rule, we have 
\begin{align*}
\int_{[0,x]}d(kQ)f&=\int_{[0,x]\setminus S_N(\mu)}d(kQ)f+\int_{[0,x]\cap S_N(\mu) }d(kQ)f\\
&=\int_{[0,x]\setminus S_N(\mu)}d(k)Qf+\int_{[0,x]\setminus S_N(\mu)}k(dQ)f+\int_{[0,x]\cap S_N(\mu) }d(kQ)f
\end{align*}
Observe that the second term
\[\int_{[0,x]\setminus S_N(\mu)}k(dQ)f=zJ\int_{[0,x]\setminus S_N(\mu)}kQfds=zJ\int_0^xkQfds\]
and the third term
\[\int_{[0,x]\cap S_N(\mu)}d(kQ)f=\sum\limits_{s\in [0,x]\cap S_N(\mu)}(k(s)-k(s-))Q(s)f(s)=\int_{[0,x]\cap S_N(\mu) }d(k)Qf\]
Hence, it follows that 
\begin{equation}
\int_{[0,x]}d(kQ)f=\int_{[0,x]}d(k)Qf+zJ\int_0^xkQfds    
\end{equation}
Moreover, by the definition of $y$, we get
\begin{equation}
k(x)Q(x)f(x)-y(0)=-J(\int_{0}^{x}kg(\mu J)d\mu Qf-z\int_0^xkQfdt)    
\end{equation}
Combining (6)-(9), we get
\begin{align*}
-\int_0^xkJg(\mu J)d\mu Qf
&=\int_{[0,x]}d(k)Qf+\int_{[0,x]}k(s-)Q(s)df(s)\\
&=-\int_0^x\frac{e^{t(s)}-e^{-t(s)}}{2t(s)}kJd\mu Qf-\int_0^x\frac{e^{t(s)}+e^{-t(s)}-2}{2t^2(s)}k\mu\{s\}d\mu Qf
\end{align*}
Let's investigate the measure $dk\cdot I+\frac{e^{t(s)}+e^{-t(s)}-2}{2t^2(s)}k\mu\{s\}d\mu$.
\par
We have
\begin{align*}
&\int_0^x(dk\cdot I+\frac{e^{t(s)}+e^{-t(s)}-2}{2t^2(s)}k\mu\{s\}d\mu)\\&=(k(x)-1)I+\sum\limits_{s\in [0,x]\cap S_N(\mu)}k(s)\frac{e^{t(s)}+e^{-t(s)}-2}{2}I\\
&=(k(x)-1+\int_0^xd\omega k)I=0
\end{align*}
Also, notice that $k$ and $f$ are continuous at 0, hence we have 
\[\int_{[0,x]}d(k)Qf=-\int_0^x\frac{e^{t(s)}+e^{-t(s)}-2}{2t^2(s)}k\mu\{s\}d\mu Qf\]
Thus
\[\int_{[0,x]}k(s-)Q(s)df(s)=-\int_0^x\frac{e^{t(s)}-e^{-t(s)}}{2t(s)}kJd\mu Qf\]
Recall $k(x-)=(1+\omega\{x\})k(x)$, we get
\[f(x)-f(0)=\int_0^xdf=-\int_0^x\frac{e^{t(s)}-e^{-t(s)}}{t(s)(e^{t(s)}+e^{-t(s)})}(Q^{-1}Jd\mu Q)f\]
This one, with $f(0)=\frac{1}{2}\begin{pmatrix} 
-i & 1\\
i & 1\\
\end{pmatrix}C$, gives the desired equation.
\end{proof}
Thanks to Lemma \ref{3.3}, we can focus on equation (6) rather than the original equation (3). For convenience, we still use $\mu$ to denote the truncation of $\mu$ on $[0,N]$, i.e., $\mu\in DS(N)$. The dB function of $B(E_\delta)$ can be written as 
\[E_\delta(z)=2k(\delta)e^{-iz\delta}f_2(\delta,z)\]
where $f$ is from Lemma \ref{3.3} with $f(0,z)=\frac{1}{2}\begin{pmatrix} 
					1\\
					1\\
\end{pmatrix}$ and $\delta\in(0,N]$.
\par
\begin{lemma}
\label{l3.4}
 There is $\delta \in (0,N]$ such that $B(E_\delta(z))=PW_\delta$ as sets. 
\end{lemma}
\begin{proof}
We focus on a small interval $[0,\delta]$ with $||\mu||((0,\delta]):=\max\limits_{i=1,2}(|\mu_i|((0,\delta]))<\frac{1}{8}$. By applying iteration on $[0,\delta]$ by putting
\[f_0=\frac{1}{2}\begin{pmatrix} 
					1\\
					1\\
\end{pmatrix},\ f_{n+1}=\frac{1}{2}\begin{pmatrix} 
					1\\
					1\\
\end{pmatrix}+\int_0^x\begin{pmatrix} 
					0 & e^{-2izs}dP(s)\\
					e^{2izs}d\overline{P}(s) & 0\\
\end{pmatrix}f_n(s)\]
we conclude that  the solution can be written as 
\[f(x,z)=\frac{1}{2}\begin{pmatrix} 
					1\\
					1\\
\end{pmatrix}+\sum_{n\geq1}v_n(x,z)\]
where
\begin{equation*}
v_n(x,z)=\left\{\begin{array}{rcl}
\frac{1}{2}\int_0^x\int_0^{t_1}\cdots\int_0^{t_{n-1}}\begin{pmatrix} 			dP(t_1)d\overline{P}(t_2)\cdots dP(t_n)e^{-2iz(t_1-t_2+t_3-\cdots+t_n)}\\
d\overline{P}(t_1)dP(t_2)\cdots d\overline{P}(t_n)e^{2iz(t_1-t_2+t_3-\cdots+t_n)}\\
\end{pmatrix}\\
\frac{1}{2}\int_0^x\int_0^{t_1}\cdots\int_0^{t_{n-1}}\begin{pmatrix} 		d\overline{P}(t_1)dP(t_2)\cdots d\overline{P}(t_n)e^{-2iz(t_1-t_2+t_3-\cdots-t_n)}\\
dP(t_1)d\overline{P}(t_2)\cdots dP(t_n)e^{2iz(t_1-t_2+t_3-\cdots-t_n)}\\
\end{pmatrix}
\end{array}\right.,
\end{equation*}
and the first (second) branch is for odd (even) $n$.
\par
Indeed, we denote $v_n$ by $\begin{pmatrix} 
					v_{n,1}\\
					v_{n,2}\\
\end{pmatrix}$. Since $v_{n,2}=v^{\#}_{n,1}$, it's enough to check the convergence of $\sum\limits_{n\geq1}v_{n,2}(x,z)$. Notice that $0<t_1-t_2+t_3-\cdots\pm t_n\leq\delta$ and $\frac{e^{t(s)}-e^{-t(s)}}{t(x)(e^{t(s)}+e^{-t(s)})}\leq 1$ for $t\geq 0$, it follows that  
\begin{align*}
|v_{n,2}|&\leq 
\frac{e^{2|z|N}}{2}\int_0^\delta\int_0^{t_1}\cdots\int_0^{t_{n-1}}d|P|(t_n)d|P|(t_{n-1})\cdots d|P|(t_1)\\&\leq\frac{e^{2|z|N}}{2}(\int_0^\delta d|P|(t))^n\leq\frac{e^{2|z|N}}{2}\cdot\frac{1}{4^n}
\end{align*}
hence we conclude that $\sum\limits_{n\geq1}v_n(x,z)$ converges uniformly on $[0,\delta]$ and on a compact set of $\mathbb{C}$.
\par
Let $s=2(t_1-t_2+t_3-\cdots\pm t_n)$. We just consider $v_{n,2}(x,z)$ when $n$ is odd. Write 
\[\Delta=\{(t_1,t_2,\cdots,t_n): 0< t_n\leq t_{n-1}\leq \cdots \leq t_1\leq \delta\}\]
We have
\[v_{n,2}(\delta,z)=\frac{1}{2}\int_{\mathbb{R}^n}d(\overline{P}\times P\cdots\times\overline{P})(t)e^{2iz(t_1-t_2+t_3-\cdots+t_n)}\chi_\Delta(t)\]
where $t=(t_1,t_2,\cdots,t_n)$.
\\
The transform 
\[T: \mathbb{R}^n\to \mathbb{R}^n, (t_1,t_2,\cdots,t_n)\mapsto (s,t_2,\cdots,t_n)\]
gives the pushforward of measure $\overline{P}\times P\cdots\times\overline{P}$,denoted by $T_*(\overline{P}\times P\cdots\times\overline{P})$, namely,
\begin{equation}
    v_{n,2}(\delta,z)=\frac{1}{2}\int_{\mathbb{R}^n}d(T_*(\overline{P}\times P\cdots\times\overline{P}))(x)e^{izs}\chi_{T(\Delta)}(x)
 \end{equation}   
where $x=(s,t_2,\cdots,t_n)$.
\par
Let $\partial T(\Delta)(s)$ be the truncation of $T(\Delta)$ at $s$, i.e., the collection of all points in $T(\Delta)$ such that the first coordinate is $s$. We define a measure on $\mathbb{R}$ by 
\[\mu_n(A):=\int_A\chi_{[0,2\delta]}(s)\int_{\partial T(\Delta)(s)}d(T_*(\overline{P}\times P\cdots\times\overline{P}))(s,t_2,\cdots,t_n)\]
This measure is compactly supported on $[0,2\delta]$ and $\mu_n\in \mathcal{M}^b(\mathbb{R})$, actually, we have 
\begin{align*}
|\mu_n|(\mathbb{R})&\leq 
\int_0^\delta\int_0^{t_1}\cdots\int_0^{t_{n-1}}d|P|(t_n)d|P|(t_{n-1})\cdots d|P|(t_1)\leq\frac{1}{4^n}
\end{align*}
Moreover, (10) can be written as 
\[v_{n,2}(\delta,z)=\frac{1}{2}\int_{\mathbb{R}}e^{izs}d\mu_n(s)\]
\par
This conclusion is also true if $n$ is even. The total variation of $\mu_n$ implies that the sequence of complex measures $\{\sum\limits_{i=1}^n\mu_i\}_{n=1}^\infty$ converges in $\mathcal{M}^b(\mathbb{R})$, and moreover, the limit, which is denoted by $\rho$, is compactly supported on $[0,2\delta]$ with $|\rho|(\mathbb{R})\leq\frac{1}{2}$. Now, the dB function becomes clear: since we have
\[\sum_{n=1}^\infty v_{n,2}(\delta,z)=\frac{1}{2}\lim_{k\to \infty}\int_{\mathbb{R}}e^{izs}d\sum\limits_{n=1}^k\mu_n(s)=\frac{1}{2}\int_{\mathbb{R}}e^{izs}d\rho(s)=\frac{\widehat{\rho}(z)}{2}\]
it follows that 
\begin{equation}
 E_\delta(z)=k(\delta)e^{-iz\delta}(1+\widehat{\rho}(z))   
\end{equation}
for some $\rho\in \mathcal{M}^b(\mathbb{R})$ and supported by $[0,2\delta]$.
\par
On the closed upper half plane, we have 
$|\widehat{\rho}(z)|\leq |\rho|(\mathbb{R})\leq \frac{1}{2}$, therefore, it follows that  
\[\frac{|k(\delta)|}{2}|e^{-iz\delta}|\leq|E_\delta(z)| \leq \frac{3|k(\delta)|}{2}|e^{-iz\delta}|\]
on $\overline{\mathbb{C}+}$.
\par
This estimation, with the fact that $\frac{|k(\delta)|}{2}>0$, gives $B(E_\delta(z))=B(e^{-iz\delta})$ as sets. With the famous result saying that $B(e^{-iz\delta})=PW_\delta$ as sets (see \cite{1}), we conclude that $B(E_\delta(z))=PW_\delta$ as sets if $\delta$ is small enough.
\end{proof}
Now, we pivot to $PW_N$ and $B(E_N)$. 
\begin{lemma}
 $PW_N= B(E_N)$ as sets.
\end{lemma}
\begin{proof}
To prove $PW_N\subset B(E_N)$, We first observe that the type of $E_L$, denoted by $\tau_L$, can be determined by $\tau_L=L$ for $L\in [0,\infty)$. Assume $|h|\leq N-\delta$ and $F=\Hat{f}\in PW_{\delta}=B(E_\delta)$, then $e^{ihz}F=\widehat{\delta_{h}*f}\in B(E_N)$. Since $f$ runs in $L^2(-\delta,\delta)$, then $\delta_{h}*f$ runs in $L^2(-\delta-h,\delta-h)$. If we pick up some proper $h$, $L^2(-N,N)$ can be decomposed into the direct sum of finitely many such subspaces. As $B(E_N)$ is Hilbert, we conclude that $PW_N\subset B(E_N)$.
\par
On the other hand, to show $B(E_N)\subset PW_N$, we need to show that $|E_N(t)|<C$ on $\mathbb{R}$ for some constant $C$, or equivalently, $f_2(N,t)$ is bounded first. Since $\mu \in DS$, we can pick up finitely many point masses so that the rest of point masses is small enough; moreover,  Moreover, as $[0,N]$ is compact, we can find out finitely many points, denoted by $\delta_0=0<\delta _1<\delta _2< \cdots<\delta _{M-1}<\delta_M=N$ so that $||\mu||((\delta_n,\delta_{n+1}))=\max\limits_{i=1,2}(|\mu_i|((\delta_n,\delta_{n+1})))<\frac{1}{8}$. We need to be aware of that $\delta_n$ can be a point mass which has a large weight, and this is why we consider open intervals rather than closed ones. 
\par
Now, we repeat the calculation of Lemma \ref{l3.4} on intervals $I_n:=(\delta_n,\delta_{n+1})$ separately.  On $I_0$, this is indeed what we did at the beginning of the proof, and the conclusion is that $f(\delta_1-,t)$ is bounded. Recall that the point $\delta_1$ will update $f(\delta_1-,t)$ to $f(\delta_1,t)$ by a constant matrix ( which is not the identity if $\delta_1$ is a point mass), thus $f(\delta_1,t)$ is still bounded. On $I_n$, we consider the iteration as above by putting:
\[f_0=f(\delta_n,z),\ f_{n+1}=f(\delta_n,z)+\int_{\delta_n}^x\begin{pmatrix} 
					0 & e^{-2izs}dP(s)\\
					e^{2izs}d\overline{P}(s) & 0\\
\end{pmatrix}f_n(s)\]
An analogous calculation shows that $f(\delta_{n+1}-,t)$ is bounded if $f(\delta_n,t)$ is so, and so as $f(\delta_{n+1},t)$. It follows from induction that $|E_N(t)|<C$ on $\mathbb{R}$ for some constant $C$.
\par
Pick $F\in B(E_N)$, then recall the type of $E_N$ is given by $\tau_N=N$, hence $B_{\tau_N}:=\{F\in B(E_N):\tau (F)\leq N\}=B(E_a)$, where $a=\max\{x\in \mathbb{R}:\tau_x(E)=x\leq N\}=N$, namely, $\tau (F)\leq N$. Moreover, \[\int_{\mathbb{R}}|F(t)|^2dt=\int_{\mathbb{R}}|\frac{F(t)}{E_N(t)}|^2|E_N(t)|^2dt<C^2||\frac{F(z)}{E_N(z)}||_{H^2}\] 
Since $\frac{F}{E_N}\in H^2$ by the definition of $B(E_N)$, it follows that $F(t) \in L^2(\mathbb{R})$. Now, the Paley-Wiener theorem implies that $B(E_N)\subset PW_N$.
\end{proof}
Finally, if $d\mu$ is absolutely continuous with respect to Lebesgue measure, then due to Volpert's chain rule, the estimation of $|v_{n,2}|$ in Lemma \ref{l3.4} can be improved on $[0,N]$ as
\begin{align*}
|v_{n,2}|&\leq 
\frac{e^{2|z|N}}{2}\int_0^N\int_0^{t_1}\cdots\int_0^{t_{n-1}}d|P|(t_n)d|P|(t_{n-1})\cdots d|P|(t_1)\\&=\frac{e^{2|z|N}}{2}\int_0^N\int_0^{t_1}\cdots(\int_0^{t_{n-2}}\int_0^{t_{n-1}}d|P|(t_n)d|P|(t_{n-1}))\cdots d|P|(t_1)\\&=\frac{e^{2|z|N}}{2}\int_0^N\int_0^{t_1}\cdots(\int_0^{t_{n-3}}\frac{1}{2}(\int_0^{t_{n-2}}d|P|(t_n))^2d|P|(t_{n-2}))\cdots d|P|(t_1)
\\&=\frac{e^{2|z|N}(|P|[0,N])^n}{2n!}
\end{align*}
\par
Moreover, we have $d\mu_n\ll dt$ since $dP\ll dt$, thus we get the conclusion that:
\[E_N(z)=e^{-izN}(1+\widehat{\rho}(z))\]
for some $\rho \in L^1(0,2N)$ (formally, $\rho$ is absolutely continuous, but we still use this notation).
\par
Recall that $E_N(z)\neq 0$, otherwise $u(N,z)=0$. A calculation shows, for $t\in \mathbb{R}$, that 
\[\frac{1}{|E_N(t)|^2}=\frac{1}{1+\Hat{g}(t)}\]
where $g:=\rho+\rho_r+\rho*\rho_r$ with $\rho_r(t)=\overline{\rho}(-t)$. By Young's convolution inequality, we conclude that $g\in L^1(\mathbb{R})$, and  Wiener's lemma then implies that there is a $\phi \in L^1(\mathbb{R})$ such that $\frac{1}{1+\Hat{g}(t)}=1+\Hat{\phi}$, this gives 
\[\frac{1}{|E_N(t)|^2}=1+\Hat{\phi}\]
Moreover, if we take the complex conjugate on both sides, we have $\Hat{\phi}=\overline{\Hat{\phi}}$. Recall that $\overline{\Hat{\phi}}=\Hat{\phi_r}$, and this gives $\phi=\phi_r$. By the definition, we have
\[\lVert F\rVert^2_{B(E_N)}=\frac{1}{\pi}\int_\mathbb{R}|F|^2\frac{dt}{|E_N|^2}=\frac{1}{\pi}\int_\mathbb{R}|\Hat{f}|^2dt+\frac{1}{\pi}\int_\mathbb{R} \overline{\hat{f}}\cdot \widehat{\phi*f}dt\]
Parseval's identity implies that 
\[\lVert F\rVert^2_{B(E_N)}=2\lVert f\rVert_{L^2}^2+2\langle f,\phi*f\rangle\]
\vspace{2cm}
\section{Reconstruction from PW Spaces}
\par
In this section, we discuss how to recover a Dirac operator from the Hilbert space $(PW_x, \lVert \cdot\rVert_{\phi,x})$ for $\phi\in \Phi_N$ and $x\in (0,N]$.
\par
In the sequel, we fix $N>0$, and assume $\phi\in \Phi_N$ and $x\in (0,N]$. We define an operator, denoted by $T^x_\phi$, as follows:
\[T^x_\phi:L^2[-x,x]\to L^2[-x,x], T^x_\phi f=\int_{-x}^x\phi(t-s)f(s)ds\]
Since $T^x_\phi f$ is interpreted as $\phi*f$ on $[-x,x]$, the definition makes sense because of Young's inequality for integral operators.
\begin{lemma}
 $T^x_\phi$ is self-adjoint and compact.   
\end{lemma}
\begin{proof}
Let $f,g\in L^2[-x,x]$, then it follows from Fubini theorem that 
\[\langle f,T^x_\phi g\rangle=\int_{-x}^x\int_{-x}^x\overline{f}(t)\phi(t-s)dtg(s)ds\]
Recall that $\phi(x)=\overline{\phi}(-x)$, then the left-hand side is indeed $\langle T^x_\phi f,g\rangle$, and this shows that $T^x_\phi$ is self-adjoint.
\\
Pick up $\phi_n\in C^{\infty}_c(-2N,2N)$ so that $\phi_n\to \phi$ in $L^1(-2N,2N)$. Since $T^x_{\phi_n}$(we abuse this notation even though $\phi_n$ is probably not in $\Phi_N$) are Hilbert-Schmidt integral operators, hence are compact as well.
Young's inequality implies that 
\[\lVert T^x_{\phi_n}f-T^x_\phi f\rVert_{L^2[-x,x]}=\lVert T^x_{\phi_n-\phi}f \rVert_{L^2[-x,x]}\leq \lVert\phi_n-\phi\rVert_{L^1(-2N,2N)}\lVert f\rVert_{L^2[-x,x]}\]
and this estimation shows that $T^x_{\phi_n}\to T^x_\phi$ in $B(L^2[-x,x])$. As a result, $T^x_\phi$ is also compact.    
\end{proof}
For $F=\hat{f}, H=\hat{h} \in PW_x$, we define an inner product (which is easy to verify) as follows:
\[[F,H]_{\phi,x}:=2\langle f,(1+T^x_\phi)h\rangle\]
The norm is given by 
\[\lVert F\rVert ^2_{\phi,x}=[F,F]_{\phi,x}\]
\begin{lemma}
 $(PW_x,\lVert\cdot\rVert_{\phi,x})$ is a Hilbert space.   
\end{lemma}
\begin{proof}
  Pick $f\neq0\in L^2[-x,x]$, and the extension of $f$ (set 0 out of $[-x,x]$) in $L^2[-N,N]$ is also denoted by $f$ for convenience, then 
  \[\langle f,(1+T^N_\phi)f\rangle=\langle f,(1+T^x_\phi)f\rangle>0\]
  Since $T^x_\phi$ is self-adjoint and compact, we have $\sigma(T^x_\phi)=\sigma_p(T^x_\phi)\cup\{0\}$ and 0 is the only accumulation point, namely,
\[\sigma(1+T^x_\phi)=\sigma_p(1+T^x_\phi)\cup\{1\}=(1+\sigma_p(T^x_\phi))\cup\{1\}\]
From the inequality above, we conclude that if $\lambda\in \sigma_p(T^x_\phi+1)$, then $\lambda>0$, otherwise there must be a $f\in L^2[-x,x]$ so that 
\[\langle f,(1+T^x_\phi)f\rangle=\lambda\lVert f\rVert^2\leq0\]
Since 1 is the only accumulation point of $\sigma_p(1+T^x_\phi)$, then there are at most finitely many eigenvalues in $[0,1-\epsilon]$ for any $\epsilon>0$. Let's denote the smallest eigenvalue by $\lambda_{min}>0$, the spectral theorem shows that 
\[\langle f,(1+T^x_\phi)f\rangle=\int_{\sigma(1+T^x_\phi)}td\lVert E_{\phi,x}(t)f\rVert^2\geq \lambda_{min}\lVert f\rVert^2\]
where $E_{\phi,x}(t)$ is the spectral family of $1+T^x_\phi$.
\par
On the other hand, we have $\langle f,(1+T^x_\phi)f\rangle\leq \lVert1+T^x_\phi\rVert\cdot\lVert f\rVert^2$. Those two inequalities together imply that 
\[\lambda_{min}\lVert f\rVert^2\leq \lVert F\rVert_{\phi,x}^2\leq \lVert1+T^x_\phi\rVert\cdot\lVert f\rVert^2 \]
It follows from this inequality that $(PW_x,||\cdot||_{\phi,x})$ is Hilbert.
\end{proof}
\begin{corollary}
\label{c4.1}
Assume $\phi\in \Phi_N$. There are $\phi_n\in C^{\infty}_c(-2N,2N)\cap\Phi_N$ such that  $\phi_n\to \phi$ in $L^1(-2N,2N)$.    
\end{corollary}
  \begin{proof}
It's trivial to see that at least for $x>0$ we can find a sequence $\phi_n$ from $C^{\infty}_c(0,2N)$ to approximate $\phi$, then with the help of $\phi(x)=\bar{\phi}(-x)$, the sequence can be extended to $(-2N,0)$ and keeps the same relation. Since we know $\langle f,(1+T^N_\phi)f\rangle\geq \lambda_{min}\lVert f\rVert^2$ and $\lvert \langle f,(T_{\phi_n}^N-T^N_\phi)f\rangle\rvert\leq \lVert\phi_n-\phi\rVert_{L^1(-2N,2N)}\lVert f\rVert^2$, so $\phi_n$ can be chosen as $\langle f,(1+T^N_{\phi_n})f\rangle>0$.
\end{proof}
Moreover, the following observation gives a more accurate description of  $(PW_x,\lVert\cdot\rVert_{\phi,x})$:
\begin{lemma}
 $(PW_x,\lVert\cdot\rVert_{\phi,x})$ is a regular dB space.   
\end{lemma}
The proof is super straightforward; one may check Theorem \ref{3.1} term by term with the help of the Paley-Wiener theorem. We leave the proof to readers (or see \cite{12}).
\subsection{the Canonical System Derived from $PW_N$ I}
As a regular dB space, there is a dB function, denoted by $E_x$, such that $(PW_x,\rVert\cdot\rVert_{\phi,x})=B(E_x)$; especially, without loss of the generality, we can normalize $E_N$ s.t. $E_N(0)=1$. It follows from Theorem \ref{3.3} that there is a canonical system $H$ with $trH=1$ on $(0,N)$ such that
\[E_N(z)=u_1(N,z)-iu_2(N,z)\]
and the corresponding reproducing kernels $J_w$ are given (Theorem 4.6 in \cite{1}) by
\[J_w(z)=\int_0^Nu^*(t,w)H(t)u(t,z)dt\]
In the sequel, we simply denote $(PW_x,||\cdot||_{\phi,x})$ by $PW_x$, and by $B_x$ the regular dB space of $u'=zJHu$ on $(0,x)$.
\par
For $PW_x$, we have $PW_{x_1}\subset PW_{x_2}$ if $x_1\leq x_2$; For $B_x$, we still have $B_{x_1}\subset B_{x_2}$ if $x_1\leq x_2$ and $x_1$ is regular, which means $H$ cannot be written as $h(x)P_\alpha$ where $h(x)$ is a scaler function and $P_\alpha$ is a constant matrix (this is called singular). See Section 10 in \cite{8} for details and Section 1.2 in \cite{1}.
\par
We denote all regular values of the canonical system on $[0,N]$ by $R$. The Ordering Theorem implies that either $PW_x\subset B_t$ or $B_t\subset PW_x$ for $t\in R$. Define for $t\in R$ a function $x(t)$ by
\[x(t)=\inf\{x\in [0,N]: B_t\subset PW_x\}\]
It is clear that $x(t)$ is increasing and that $x(0)=0,x(N)=N$. We apply a modification: if $(0,N)$ starts with a singular interval $(0,a)$ and $E_a(z)=1$, we delete this initial interval and rescale the rest so that we still have a canonical system on $(0,N)$. Clearly, this modification does not change $B_N$.
\begin{lemma}
 $\forall x\in (0,N)$, $PW_x=\bigcap\limits_{y>x}PW_y=\overline{\bigcup\limits_{y<x}PW_y}$. This closure is taken in $PW_N$.   
\end{lemma}
\begin{proof}
 Since $PW_x\subset PW_y$ if $x\leq y$, it's clear that $PW_x\subset \bigcap\limits_{y>x}PW_y$. On the other hand, let $F\in \bigcap\limits_{y>x}PW_y$, then $F\in PW_{x+\frac{1}{n}}$ for all positive integers $n$, i.e., $F(z)=\hat f_n$ where $f_n \in L^2(-(x+\frac{1}{n}),x+\frac{1}{n})$. By the uniqueness of the inverse Fourier transform, we get that all $f_n$ are the same. This is true only if $f_n$ are supported by $[-x,x]$, hence $F\in PW_x$.
\par
Since $PW_y\subset PW_x$ if $y\leq x$, hence $\overline{\bigcup\limits_{y<x}PW_y}\subset PW_x$ as $PW_y$ are closed in $PW_x$. Let $F\in PW_x$, then $F=\hat{f}$ for some $f\in L^2(-x,x)$. Define $f_n=\chi_{(-(x-\frac{1}{n}),x-\frac{1}{n})}f$, then $F_n\to F$ in $PW_x$, hence it follows $F\in \overline{\bigcup\limits_{y<x}PW_y}$.   
\end{proof}
\begin{lemma}
\label{l4.5}
 The (modified) canonical system $u'=zJHu$ has no singular points. Moreover, $PW_{x(t)}=B_t$ for all $t\in [0,N]$.   
\end{lemma}
\begin{proof}
 If $t\in R$, then either $PW_x\subset B_t$ or $B_t\subset PW_x$ for $t\in R$, hence it follows that $\overline{\bigcup\limits_{y<x(t)}PW_y}\subset B_t\subset \bigcap\limits_{y>x(t)}PW_y$, namely,  $PW_{x(t)}=B_t$.  
 \par
 If $(a,b)$ is a singular interval, then for regular values $a,b$, we have $PW_{x(a)}=B_a,PW_{x(b)}=B_b$, i.e., 
\[PW_{x(b)}\ominus PW_{x(a)}=B_b\ominus B_a\]
Corollary 10.11 in \cite{8}  implies that the right-hand side is one-dimensional; however, the left-hand side cannot be one-dimensional. This contradiction shows that there are no singular points.
\end{proof}
\subsection{Two Integral Equations}
As regular dB spaces, we can discuss reproducing kernels and conjugate kernels, and those two different types of kernels will give us two equations that can be used to reconstruct the canonical system $H$ (equivalently, the Dirac system).
\par
\subsubsection{Reproducing Kernels for $w=0$ in $PW_x$}
Let's denote reproducing kernels for $w=0$ in $PW_x$ by $J_0(x,z)$, then there exist $j(x,t)\in L^2(-x,x)$ such that $J_0(x,z)=\int_{-x}^xj(x,t)e^{izt}dt$. Since $\forall F=\int f e^{izt}dt\in PW_x$, we have $[J_0,F]_{\phi,x}=F(0)$, therefore it follows
\[F(0)=\int_{-x}^xfdt=\langle1,f\rangle_{L^2(-x,x)}=2\langle j,(1+T_\phi^x)f\rangle\]
As the operator $1+T_\phi^x$ is self-adjoint, we actually have
\[\langle1,f\rangle_{L^2(-x,x)}=\langle2(1+T_\phi^x)j,f\rangle\]
Since $f$ is arbitrary in $L^2(-x,x)$, we conclude that 
\[(1+T_\phi^x)j=\frac{1}{2}\]
or equivalently,
\begin{equation}
j(x,t)+\int_{-x}^x\phi (t-s)j(x,s)ds=\frac{1}{2} 
\end{equation}
on $t\in [-x,x]$.
\subsubsection{Conjugate Kernels for $w=0$ in $PW_x$}
The signal function, denoted by $sgn(x)$, is the function defined by
\[sgn(x):=\left\{\begin{array}{rcl}
1,&&x\in [0,\infty)\\
-1,&&x\in (\infty,0)
\end{array}\right.\]
For $\phi \in \Phi_N$, we define a function $\Phi(x)$ by
\[\Phi(x)=\int_0^x\phi(s)ds\]
we also follow the convention: if $x<0$, then $\int_0^x\phi(s)ds=-\int^0_x\phi(s)ds$. Also notice that we have the property $\overline{\Phi}(x)=-\Phi(-x)$.
\par
We define a function by $\psi(s)=(2\Phi(s)+sgn(s))i$, then 
define a bounded linear functional for $F=\widehat{
f}\in PW_x$ as follows:
\[\widehat{F}(0):=\int_{-x}^x\overline{\psi}(t)f(t)dt\]
The following lemma is used later:
\begin{lemma}
\label{l4.6}
For all $F=\widehat{f},G=\widehat{g}\in PW_x$,we have 
\[\widehat{F}(0)\overline{G(0)}-F(0)\overline{\widehat{G}(0)}=[S_0G,F]_{\phi,x}-[G,S_0F]_{\phi,x}\]    \end{lemma}
\begin{proof}
Let's write $s_0F(t):=I_f(t)-\chi_{(0,x)}(t)F(0)$, where $I_f(t)=\int_{-x}^tf(t)dt$.   Notice that $I_f(x)=F(0)$, then it is easy to show that $S_0F(z)=-i\widehat{s_0F}(z)$ as Fourier transform, as a consequence, we have
\[[S_0G,F]_{\phi,x}=2i\langle s_0G,(1+T_\phi^x)f\rangle\]
\[[G,S_0F]_{\phi,x}=-2i\langle(1+T_\phi^x)g,s_0F\rangle\]
Notice that
\[\langle(1+T_\phi^x)I_g,f\rangle=\overline{(1+T_\phi^x)I_g}\cdot I_f|_{-x}^x-\langle (1+T_\phi^x)g,I_f\rangle+\overline{G(0)}\int_{-x}^x\phi(x-t)I_f(t)dt\]
and 
\[\langle I_g,(1+T_\phi^x)f\rangle=\langle(1+T_\phi^x)I_g,f\rangle\]
We conclude that 
\[\langle I_g,(1+T_\phi^x)f\rangle+\langle (1+T_\phi^x)g,I_f\rangle=F(0)\overline{(1+T_\phi^x)I_g}(x)+\overline{G(0)}\int_{-x}^x\phi(x-t)I_f(t)dt\]
Also, since $\langle (1+T_\phi^x)g, I_f\rangle=\langle g,(1+T_\phi^x)I_f\rangle$, the analogous equation holds as well once we apply the integration by part to $\langle g,(1+T_\phi^x)I_f\rangle$:
\[\langle I_g,(1+T_\phi^x)f\rangle+\langle (1+T_\phi^x)g,I_f\rangle=\overline{G(0)}(1+T_\phi^x)I_f(x)+F(0)\int_{-x}^x\phi(t-x)\overline{I_g(t)}dt\]
Therefore, we get
\[F(0)\overline{(1+T_\phi^x)I_g}(x)+\overline{G(0)}\int_{-x}^x\phi(x-t)I_f(t)dt=\overline{G(0)}(1+T_\phi^x)I_f(x)+F(0)\int_{-x}^x\phi(t-x)\overline{I_g(t)}dt\]
Notice that
\[T_\phi^xI_f(x)=\int_{-x}^x\phi(x-t)I_f(t)dt=\int_{-x}^x\Phi(x-t)f(t)dt\]
\[\langle (1+T_\phi^x)\chi_{(0,x)},f\rangle=\langle \chi_{(0,x)},f\rangle+\int_{-x}^x(\Phi(x-t)-\Phi(-t))f(t)dt\]
A tedious calculation then shows that 
\[[S_0G,F]_{\phi,x}-[G,S_0F]_{\phi,x}=\widehat{F}(0)\overline{G(0)}-F(0)\overline{\widehat{G}(0)}\]
\end{proof}
\par
Now, let's pivot back to $\widehat{F}(0)$. As a bounded linear functional on $PW_x$, the Rieze representation theorem guarantees that there must be a unique $K_0(x,z)= \int_{-x}^xk(x,t)e^{izt}dt\in PW_x$ so that $[K_0,F]_{\phi,x}=\widehat{F}(0)$ for all $F=\widehat{f}\in PW_x$. In other words, we get 
\[\langle2(1+T_\phi^x)k,f\rangle=\int_{-x}^x\overline{\psi}(t)f(t)dt\]
i.e., $(1+T_\phi^x)k=\frac{1}{2}\psi$, or equivalently,
\begin{equation}
k(x,t)+\int_{-x}^x\phi (t-s)k(x,s)ds=\frac{1}{2} 
\psi(t)\end{equation}
on $t\in [-x,x]$.
\subsubsection{Regularity Properties of the Integral Equations}
Those two integral equations (12) and (13) can be summarized as the following integral equation:
\begin{equation}
    p(x,t)+\int_{-x}^x\phi (t-s)p(x,s)ds=g(x,t)
\end{equation}
for $t\in [-x,x]$ in some proper spaces; moreover, if $x=0$, we define $p(0,0)=g(0,0)$.
\par
To simplify our notations, let's define a triangle with a number $m\in (0,N]$ by
\[\Delta_m:=\{(x,t)\in \mathbb{R}^2:0<|t|<x<m\}\]
It is clear that if $g(x,t)=\frac{1}{2}$ and $g(x,t)=\frac{1}{2}\psi(t)$, we go back to (12) and (13), respectively.
\par 
We define three operators $L^x_\phi,C^x_\phi$ and $H^x_\phi$ (recall $T^x_\phi$) for $x\in(0,N]$:
\begin{definition}
 \[L^x_\phi:L^1[-x,x]\to L^1[-x,x], L^x_\phi f:=\int_{-x}^x\phi(t-s)f(s)ds\] 
 \[C^x_\phi:C[-x,x]\to C[-x,x], C^x_\phi f:=\int_{-x}^x\phi(t-s)f(s)ds\]
Here, $C[-x,x]$ is the Banach space of all continuous functions endowed with the supremum norm.
\end{definition}
As before, Young's inequality ensures that them are well-defined, and since $C[-x,x]\subset L^2[-x,x]\subset L^1[-x,x]$ as sets, if we pick up some function $f\in C[-x,x]$ for instance, then we do have $T^x_\phi f=L^x_\phi f=C^x_\phi f$.
\par
Let's define 
\[H[-x,x]:=C[-x,x]\oplus L\{\chi[0,x]\}\]
where $L$ means all linear combinations.
\par
We can treat $H[-x,x]$ as the direct sum of two Banach spaces, then the scalar multiplication and vector addition can be defined as usual so that $H[-x,x]$ becomes a vector space. For $(f,a\chi[0,x])\in H[-x,x]$, we define the norm (which is easy to check) as follows:
\[\rVert(f,a\chi[0,x])\lVert_{H[-x,x]}:=\rVert f\lVert_{C[-x,x]}+|a|\]
It's easy to see that $H[-x,x]$ with the norm above is a Banach space; moreover, $C[-x,x]$ can be embedded into $H[-x,x]$ isometrically. In all, we have the chain:
\[C[-x,x]\subset H[-x,x]\subset L^2[-x,x]\subset L^1[-x,x]\]
and the second $\subset$ may be interpreted as $f+a\chi[0,x]$ for $(f,a\chi[0,x])\in H[-x,x]$.
\par
Let $f\in C[-x,x]$. We have
\[\int_{-x}^x\phi(t-s)(f(s)+a\chi[0,x](s))ds=\int_{-x}^x\phi(t-s)f(s)ds+a\int_{t-x}^t\phi(s)ds\]
It is clear that the right-hand side is continuous; hence, we can define an operator $H^x_\phi$ as follows:
\begin{definition}
  \[H^x_\phi:H[-x,x]\to H[-x,x]\]
\[H^x_\phi(f,a\chi[0,x]):=(\int_{-x}^x\phi(t-s)(f(s)+a\chi[0,x](s))ds,0)\]  
\end{definition}
The following lemma can be found in many sources (\cite{15}, for example), but we restate it here for convenience.
\begin{lemma}
\label{l4.7}
Let $X,Y$ be Banach spaces. An operator $T:X\to Y$ has a continuous inverse if and only if 
\[\gamma_T:=\inf \{\lVert Tx\rVert: x\in D(T), \lVert x\rVert\geq 1\}>0.\]
In this case, we have $\lVert T^{-1}\rVert=\gamma_T^{-1}$.    
\end{lemma}
\begin{proof}
 Let's assume $\gamma_T>0$ first. It follows that $\ker(T)=\{0\}$, hence the inverse exists. We have
\[\lVert T^{-1}\rVert=\sup\limits_{y\in D(T^{-1})}\lVert\frac{T^{-1}y}{\lVert y\rVert}\rVert=\sup\limits_{x\in D(T)}\lVert \frac{x}{\lVert Tx\rVert}\rVert=\frac{1}{\inf\limits_{x\in D(T)}\lVert T\frac{x}{\lVert x\rVert}\rVert}\]
This show that $T^{-1}$ is continuous and $\lVert T^{-1}\rVert=\gamma_T^{-1}$. 
\par
On the other hand, if $\gamma_T=0$, then either there is $x\neq 0$ such that $Tx=0$, i.e., $T$ has no inverse; or $T$ has the inverse but there is a sequence $\{x_n\}$ such that $\lVert x_n\rVert=1$ and $\lVert y_n\rVert:=\lVert Tx_n\rVert\to 0$; however, this implies that $\lVert\frac{T^{-1}y_n}{\lVert y_n\rVert}\rVert=\frac{1}{\lVert y_n\rVert}\to \infty$, i.e., $T^{-1}$ is not continuous.
\end{proof}
The following lemma is the foundation for using Fredholm theory.
\begin{lemma}
 $L^x_\phi,C^x_\phi,H^x_\phi$ are compact.   
\end{lemma}
\begin{proof}
Pick up $\phi_n\in C^{\infty}_c(-2N,2N)$ so that $\phi_n\to \phi$ in $L^1(-2N,2N)$ and $||\phi_n-\phi||_{L^1}<1$ (Corollary \ref{c4.1}),  we first claim that $C^x_{\phi_n}$ are compact.
\par
Indeed, Let $\{f_m\}\subset C[-x,x]$ s.t. $\lVert f_m\rVert_{C[-x,x]}\leq 1$. Since $\lVert C^x_{\phi_n}f_m\rVert_{C[-x,x]}\leq \lVert \phi\rVert_{L^1}+1$, we get that $\{C^x_{\phi_n}f_m\}$ is uniformly bounded. On the other hand, since
\[|C^x_{\phi_n}f_m(t)-C^x_{\phi_n}f_m(t_0)|\leq \int_{-x}^x|\phi_n(t-s)-\phi_n(t_0-s)|ds\]
As $\phi_n$ is continuous and compactly supported, it follows that $\{C^x_{\phi_n}f_m\}$ is equicontinuous. Now Arzela-Ascoli theorem says that $C^x_{\phi_n}$ is compact. Also notice that $C^x_{\phi_n}\to C^x_{\phi}$ in $B(C[-x,x])$, we get that $C^x_{\phi}$ is compact as well.
\par
The same idea can be applied to $H^x_\phi$; for $L^x_\phi$, we just need to notice that the range of $L^x_{\phi_n}$ is in $C[-x,x]$, then the same idea also shows that $L^x_\phi$ is compact.
\end{proof}
\begin{theorem}
\label{t4.1}
  $1+T^x_\phi$, $1+L^x_\phi$, $1+C^x_\phi$, and $1+H^x_\phi$ are bijections.   
\end{theorem}
\begin{proof}
 Since all operators listed are Fredholm with index 0, we just need to show that their kernels are $\{0\}$. 
 \par
 For $1+T^x_\phi$, if $f\in L^2[-x,x]$, we have 
\[\langle f,(1+T^x_\phi)f\rangle=\langle f,(1+T^N_\phi)f\rangle \geq \lambda_{(min,N)}\lVert f\rVert^2_{L^2[-N,N]}=\lambda_{(min,N)}\lVert f\rVert^2_{L^2[-x,x]}\]
where $\lambda_{(min,N)}>0$ is the smallest eigenvalue of $1+T^N_\phi$.
\par
If $\gamma$ is an eigenvalue of $1+T^x_\phi$, then for a corresponding eigenvector $f_\gamma$, we must have 
\[\langle f_\gamma,(1+T^x_\phi)f_\gamma\rangle=\gamma\lVert f_\gamma\rVert^2\geq \lambda_{(min,N)}\lVert f_\gamma\rVert^2\]
Hence, it follows that all eigenvalues of $1+T^x_\phi$ are not less than $\lambda_{(min,N)}$; and as a consequence, we have 
\[\lVert(1+T^x_\phi)f\rVert^2=\int_{\sigma(1+T^x_\phi)}t^2d\lVert E_{\phi,x}(t)f\rVert^2\geq \lambda_{(min,N)}^2 \lVert f\rVert^2\]
where $E_{\phi,x}(t)$ is the spectral family of $1+T^x_\phi$.
\par
Now, Lemma \ref{l4.7} implies that $1+T^x_\phi$ is invertible with $\lVert (1+T^x_\phi)^{-1}\rVert\leq \lambda_{(min,N)}^{-1}$, specially, $\ker (1+T^x_\phi)=\{0\}$.
\par
For $1+C^x_\phi$, since any eigenvalue of $1+C^x_\phi$ must be an eigenvalue of $1+T^x_\phi$, hence all eigenvalues of $1+C^x_\phi$ are not less than $\lambda_{(min,N)}$, i.e., 0 cannot be an eigenvalue, thus $\ker (1+C^x_\phi)=\{0\}$. Moreover, this argument works for $1+H^x_\phi$.
\par
For $1+L^x_\phi$, we change our strategy: we will show that $Ran(1+L^x_\phi)=L^1[-x,x]$, hence $dim\ker(1+L^x_\phi)=dim(coker(1+L^x_\phi))=dim(L^1\setminus Ran(1+L^x_\phi)) =0$. Indeed, Let $g\in L^1[-x,x]$, then we have $\{g_n\}\subset C[-x,x]$ so that $g_n\to g$ in $L^1[-x,x]$. Since we've already proved that $1+C^x_\phi$ is a bijection, then there must be $\{f_n\}\subset C[-x,x]$ so that $(1+C^x_\phi)f_n=g_n$; moreover, this fact implies  $(1+L^x_\phi)f_n=g_n$, i.e., $\{g_n\}\subset Ran(1+L^x_\phi)$. Since we know that $Ran(1+L^x_\phi)$ is closed (for a Fredholm operator), it follows that $g\in Ran(1+L^x_\phi)$ as the limit of $\{g_n\}$.
\end{proof}
Let's pivot back to $1+H^x_\phi$ again. 
\par
Suppose $g=(g_c,a\chi_{[0,x]})\in H[-x,x]$, then there is a unique $f=(f_c,a\chi_{[0,x]})\in H[-x,x]$ such that $(1+H^x_\phi)f=g$, i.e., 
\[f_c(x,t)+\int_{-x}^x\phi(t-s)f_c(x,s)ds=g_c(x,t)-a\int_{t-x}^t\phi(s)ds\]
This gives the unique solution $f_c+a\chi_{[0,x]} \in L^2[-x,x]$ of the equation $(1+T^x_\phi)f=g_c+a\chi_{[0,x]}$.
\par
Especially, recall the integral equation (13) and notice that 
\[\frac{1}{2}\psi(t)=(i(\Phi(s)-\frac{1}{2}),i\chi[0,x](s)) \in H[-x,x]\]
we have 
\[k(x,t)=k_c(x,t)+i\chi_{[0,x]}\]
where $k_c \in C[-x,x]$ is the solution of 
\[k_c(x,t)+\int_{-x}^x\phi(t-s)k_c(x,s)ds=i\Phi(t-x)-\frac{1}{2}i\]
The argument above implies that it is enough to assume $g$ is continuous in the equation $(14)$.
\begin{theorem}
\label{t4.2}
Assume $g(x,t) \in C(\overline{\Delta_N})$. For any $x\in [0,N]$, $p(x,t)\in C[-x,x]$  is the solution of 
\[p(x,t)+\int_{-x}^x\phi (t-s)p(x,s)ds=g(x,t)\]
We have
\begin{enumerate}
    \item $p(x,t) \in C(\overline{\Delta_N})$;
    \item Under an extra assumption that for all fixed $x\in (0,N]$, $g(x,t) \in AC[-x,x]$, we have that $p(x,t) \in AC[-x,x]$ with respect to $t$. Moreover, the partial derivative $p_t(x,t)$ satisfies
\[p_t(x,t)+\int_{-x}^x\phi(t-s)p_t(x,s)ds=g_t(x,t)+\phi(t-x)p(x,x)-\phi(t+x)p(x,-x)\]
where $g_t(x,t)$ is the partial derivative of $g$ with respect to $t$.
\end{enumerate}
\end{theorem}
\begin{corollary}
\label{c4.3}
  Assume $g(x,t)=\frac{1}{2}$ or $g(x,t)=i\Phi(t-x)-\frac{1}{2}i$.  For any $x\in (0,N]$, we denote by $p_x(x,t)$ (\textquotedblleft partial derivative\textquotedblright about $x$) the solution of
\[p_x(x,t)+\int_{-x}^x\phi(t-s)p_x(x,s)ds=g_x(x,t)-\phi(t-x)p(x,x)-\phi(t+x)p(x,-x)\]
where $g_x(x,t)$ is the partial derivative of $g$ with respect to $x$.
\begin{enumerate}
    \item $\int_{-x}^xp_x(x,s)ds$ is continuous, and $\int_{-x}^x|p_x(x,s)|ds \in L^1(0,N)$;
    \item for any $x\in [0,N]$ and $h(x)$ a bounded measurable function, we have
\[\int_{-x}^xh(s)p(x,s)ds=\int_{-x}^xh(s)p(|s|,s)ds+\int_0^xdm\int_{-m}^mh(s)p_x(m,s)ds\]
\end{enumerate}
\end{corollary}
The proofs are technical, so we postpone them. Readers can skip those two proofs temporarily and move to the next section directly. 
\subsection{the Canonical System Derived from $PW_N$ II}
Let's summarize what we have derived until now: we theoretically reconstruct a canonical system $H$ with $trH=1$ on $(0,N)$ and $PW_{x(t)}=B_t$ for all $t\in [0,N]$. We next want to show that the function $x(t)$ is redundant after a proper normalization. Most of the tools and ideas can be found in \cite{8}.
\par
We briefly introduce conjugate kernels and reproducing kernels for a canonical system; for more details, please see \cite{1} and \cite{8}. Let $H$ be a canonical system on $(0,N)$ without singular points (Lemma \ref{l4.5}), the regular dB space given by $E_N(z)=u_1(N,z)-iu_2(N,z)$ is 
\[B(E_N):=\{F(z)=Uf(z)=\int_0^Nu^*(x,\overline{z})H(x)f(x)dx:f\in L_H^2(0,N)\}\]
and the reproducing kernels are given by
\[J_w(z)=\int_0^Nu^*(x,w)H(x)u(x,z)dx\]
The conjugate kernels, $K_w(z):=\frac{v^*(N,w)Ju(N,z)-1}{z-\overline{w}}$, are given by 
\[K_w(z)=\int_0^Nv^*(x,w)H(x)u(x,z)dx=\int_0^Nu^*(x,\overline{z})H(x)v(x,\overline{w})dx\]
where $v$ is the solution of (2) with $v(0,z)=\begin{pmatrix}
    0\\
    1\\
\end{pmatrix}$.
\par
This integral form of $K_w(z)$ implies $K_w(z)\in B(E_N)$. If we define $\widetilde{F}(z):=[K_z,F]$ for $F(z)=\int_0^Nu^*(x,\overline{z})H(x)f(x)dx$, then 
\[\widetilde{F}(z)=\int_0^Nv^*(x,\overline{z})H(x)f(x)dx\]
Moreover, assume $0<N_1<N_2$ and $F\in B(E_{N_1})$, then 
\[[K_z^{(N_1)},F]_{B(E_{N_1})}=[K_z^{(N_2)},F]_{B(E_{N_2})}\]
The same as Lemma \ref{l4.6}, we have for all $F,G\in B(E_N)$, 
\[\widetilde{F}(0)\overline{G(0)}-F(0)\overline{\widetilde{G}(0)}=[S_0G,F]-[G,S_0F]\]
\par
Now, it is time to introduce this normalization:
\begin{lemma}
\label{l4.9}
  There is a canonical system $H\in L^1(0,N)$ so that for all $x\in[0,N]$ we have $PW_x=B_x$ (as de Branges spaces). Moreover, $H(x)$ can be chosen so that $\widehat{F}(0)=\widetilde{F}(0)$ for all $F\in PW_N=B_N$.  
\end{lemma}
\begin{proof}
Two identities about $\widehat{F}(0),\widetilde{F}(0)$ imply for all $F,G\in B_N=PW_N$,
 \[(\widehat{F}(0)-\widetilde{F}(0))\overline{G(0)}=F(0)(\overline{\widehat{G}(0)}-\overline{\widetilde{G}(0)})\]
 If we choose a special $G\in PW_N$ such that $G(0)\neq 0$, and write $c:=\frac{\overline{\widehat{G}(0)}-\overline{\widetilde{G}(0)}}{\overline{G(0)}}$, then we have for all $F\in B_N=PW_N$
\[\widehat{F}(0)=\widetilde{F}(0)+cF(0)\]  
Moreover, by plugging $F=G$ into this equation, we get that the constant $c$ is real.
\par
In the sequel, to avoid confusion, we temporarily denote the reproducing kernels and conjugate kernels for $w=0$ of $B_t$ by $j_0(t,z)$ and $k_0(t,z)$, respectively. We have
$[K_0,F]_{\phi,N}=[K_0,F]=[k_0(N,\cdot),F]+c[j_0(N,\cdot),F]$, or equivalently,
\[K_0(N,z)=k_0(N,z)+cj_0(N,z)\]
Moreover, if $F$ is chosen from $PW_{x(t)}=B_t\subset PW_N$, then we also have
\begin{equation}
  K_0(x(t),z)=k_0(t,z)+cj_0(t,z)  
\end{equation}
Now, we invoke the formula (15.3) in \cite{8}:
\[B_t=\overline{\bigcup\limits_{s<t}B_s}=\bigcap\limits_{s>t}B_s\]
This conclusion also implies that $x(t)$ is continuous and strictly increasing, also see \cite{8} for more details.
\par
Denote by $t(x)$ the inverse of $x(t)$; then $t(x)$ is also strictly increasing and continuous. Since we have
\[j_0(t(x),0)=\int_0^{t(x)}u^*(s,0)H(s)u(s,0)ds=\int_0^{t(x)}H_{11}(s)ds\]
\[\widetilde{k_0(t(x),\cdot)}(0)=[k_0(t(x),\cdot),k_0(t(x),\cdot)]=[Uv(\cdot,0),Uv(\cdot,0)]\]
The second equation is from the unitary map $U$ in the spectral representation theorem and the fact that $H(x)$ has no singular intervals. Therefore, it follows that 
\[\widetilde{k_0(t(x),\cdot)}(0)=[v(\cdot,0),v(\cdot,0)]_{L^2_H(0,t(x))}=\int_0^{t(x)}H_{22}(s)ds\]
If we add those two equations together, then we have 
\[t(x)=\int_0^{t(x)}trH(s)ds=j_0(t(x),0)+\widetilde{k_0(t(x),\cdot)}(0)\]
To evaluate $\widetilde{k_0(t(x),\cdot)}(0)$, we apply the functional $\tilde{F}(0)$ to equation (15) to get
\[\widetilde{k_0(t(x),\cdot)}(0)=\widetilde{K_0(x,\cdot)}(0)-c\widetilde{j_0(t(x),\cdot)}(0)\]
Also recall that $j_0(t(x),z)=J_0(x,z)$, therefore we have 
\[\widetilde{K_0(x,\cdot)}(0)=\widehat{K_0(x,\cdot)}(0)-cK_0(x,0)\]
\[\widetilde{J_0(x,\cdot)}(0)=\widehat{J_0(x,\cdot)}(0)-cJ_0(x,0)\]
If we assemble all of those facts, we get
\[t(x)=(c^2+1)J_0(x,0)-c\widehat{J_0(x,\cdot)}(0)+\widehat{K_0(x,\cdot)}(0)-cK_0(x,0)\]
More explicitly, we have by definition that 
\[t(x)=(c^2+1)\int_{-x}^xj(x,s)ds-c\int_{-x}^x\overline{\psi}(s)j(x,s)ds+\int_{-x}^x\overline{\psi}(s)k(x,s)ds-c\int_{-x}^xk(x,s)ds\]
We want to show that $t(x)$ is absolutely continuous. From Corollary \ref{c4.3}, we have for $p(x,t)=j(x,t),k_c(x,t)$ and $h(s)=1,\overline{\psi}(s)$
\[\int_{-x}^xh(s)p(x,s)ds=\int_0^xh(s)(p(s,s)+p(s,-s))ds+\int_0^xdm\int_{-m}^mh(s)p_x(m,s)ds\]
namely, $t(x)$ is indeed absolutely continuous.
\par
Finally, we can use a standard procedure (see section 1.3,\cite{1}) to transform the canonical system to another one with the desired property: Let's define $\widetilde{u}(x)=u(t(x))$ and $\widetilde{H}(x)=t'(x)H(t(x))$. Since $t'\geq 0$ a.e., hence this $\widetilde{H}$ is in $L^1(0,N)$ and $\widetilde{H}\geq 0$ a.e.. Moreover, we have 
\[J\widetilde{u}'=z\widetilde{H}\widetilde{u}\]
By the definition of $\widetilde{u}$, we also have $\widetilde{B}_x=B_{t(x)}=H_x$.
\\
Finally, we write $H_c(x)=\begin{pmatrix} 
					1 & 0\\
					c & 1 \\
\end{pmatrix}\widetilde{H}\begin{pmatrix} 
					1 & c\\
					0 & 1 \\
\end{pmatrix}$. For this canonical system, we do have $\widehat{F}(0)=\widetilde{F}(0)$.
\end{proof}
Next, we may explicitly write down this normalized $H$ with the help of function $j$ and $k$ as follows:
\begin{corollary}
\label{c4.4}
The canonical system $H$ from Lemma \ref{l4.9} is given (a.e.) by
\[H_{11}(x)=j(x,x)+j(x,-x)+\int_{-x}^xj_x(x,s)ds\]
\[H_{12}(x)=H_{21}(x)=k(x,x)+k(x,-x)+\int_{-x}^x(k_c)_x(x,s)ds\]
\[H_{12}(x)=H_{21}(x)=\overline{\psi}(x)j(x,x)+\overline{\psi}(-x)j(x,-x)+\int_{-x}^x\overline{\psi}(s)j_x(x,s)ds\]
\[H_{22}(x)=\overline{\psi}(x)k(x,x)+\overline{\psi}(-x)k(x,-x)+\int_{-x}^x\overline{\psi}(s)(k_c)_x(x,s)ds\]
\end{corollary}
\begin{proof}
Recall that \[J_0(x,0)=\int_{-x}^xj(x,t)dt=\int_0^xu^*(t,0)H(t)u(t,0)dt=\int_0^xH_{11}(t)dt\]
\[K_0(x,0)=\int_{-x}^xk(x,t)dt=\int_0^xv^*(t,0)H(t)u(t,0)dt=\int_0^xH_{12}(t)dt\]
And those two give the first two identities once we differentiate on both sides. 
\par
On the other hand, we have \[\widehat{J_0(x,z)}(0)=\widetilde{J_0(x,z)}(0)=\int_{-x}^x\overline{\psi}(t)j(x,t)dt=[K_0,J_0]_{B_x}=\int_0^xH_{12}(t)dt\]
This gives the third identity.
\par
The last one is from 
\[\widetilde{K_0(x,z)}(0)=\widehat{K_0(x,z)}(0)=[K_0,K_0]_{B_x}=\int_0^xH_{22}(t)dt\]
\end{proof}
\subsection{Continuity of the Canonical System}
The following theorem shows that for $\phi\in\Phi_N$, the canonical system we have achieved is continuous.
\begin{theorem}
\label{t4.3}
 Under the assumptions of Lemma \ref{l4.9}, and write \[A(x):=2j(x,x)j(x,-x),C(x):=2k(x,x)k(x,-x)\]\[B(x):=j(x,-x)k(x,x)+j(x,x)k(x,-x),D(x):=j(x,-x)k(x,x)-j(x,x)k(x,-x)\]
 then
 \begin{enumerate}
     \item $D\neq 0$ on $(0,N]$ and $\lim\limits_{x\to 0^+}D(x)=\frac{1}{2}i$;
     \item $H$ has a continuous representative so that
\[\left\{\begin{array}{rcl}
H_{11}=\frac{A}{D}i\\
H_{12}=\frac{B}{D}i\\
H_{22}=\frac{C}{D}i\\
\end{array}\right.\]
on $(0,N]$ everywhere, and $H(0)=I, \det(H)=1$;
\item As a consequence of 2., the corresponding Dirac system $\mu\in DS$ is a continuous measure. 
 \end{enumerate}
\end{theorem}
\begin{proof}
Let's list all the significant equations we have gotten:
\par
For $g(x,t)=\frac{1}{2}$, we have 
\[j_t(x,t)+\int_{-x}^x\phi(t-s)j_t(x,s)ds=\phi(t-x)j(x,x)-\phi(t+x)j(x,-x)\]
\[j_x(x,t)+\int_{-x}^x\phi(t-s)j_x(x,s)ds=-\phi(t-x)j(x,x)-\phi(t+x)j(x,-x)\]
For $g(x,t)=i\Phi(t-x)-\frac{1}{2}i$, we have 
\[(k_c)_t(x,t)+\int_{-x}^x\phi(t-s)(k_c)_t(x,s)ds=\phi(t-x)k(x,x)-\phi(t+x)k(x,-x)\]
\[(k_c)_x(x,t)+\int_{-x}^x\phi(t-s)(k_c)_x(x,s)ds=-\phi(t-x)k(x,x)-\phi(t+x)k(x,-x)\]
If we multiply the second and the forth equations by $k(x,-x)$ and $j(x,-x)$ respectively, and evaluate the difference, we get
\[(k(x,-x)j_x-j(x,-x)(k_c)_x)+\int_{-x}^x\phi(t-s)(k(x,-x)j_x-j(x,-x)(k_c)_x)\]
\[=\phi(t-x)(j(x,-x)k(x,x)-j(x,x)k(x,-x))\]
On the other hand, we also have
\[(j_t-j_x)+\int_{-x}^x\phi(t-s)(j_t-j_x)=2\phi(t-x)j(x,x)\]
If we compare those two equations and recall that $1+L_\phi^x$ is a bijection, we conclude that the solutions are the same up to some obvious coefficients, more precisely:
\[B(x)j_x(x,t)-A(x)(k_c)_x(x,t)=D(x)j_t(x,t)\]
If we integrate from $-x$ to $x$ on both sides and invoke Corollary \ref{c4.4}, then 
\[B(H_{11}-j(x,x)-j(x,-x))-A(H_{12}-k(x,x)-k(x,-x))=D(j(x,x)-j(x,-x))\]
namely, 
\begin{equation}
BH_{11}-AH_{12}=0    
\end{equation}
On the other hand, we integrate from $-x$ to $x$ on both sides of the following differential equation
\[B(x)\overline{\psi}(t)j_x(x,t)-A(x)\overline{\psi}(t)(k_c)_x(x,t)=D(x)\overline{\psi}(t)j_t(x,t)\]
Recall $\overline{\phi}(s)=\phi(-x)$ and equation (12), we have
\[j(x,0)+\int_{-x}^x\overline{\phi}(s)j(x,s)ds=\frac{1}{2}\]
therefore 
\[\int_{-x}^x\overline{\psi}(s)j_t(x,s)ds=i+j(x,x)\psi(-x)-j(x,-x)\psi(x)\]
It follows from this calculation that 
\begin{equation}
 BH_{12}-AH_{22}=iD   
\end{equation}
Last, we claim that
\begin{equation}
 \det H=H_{11}H_{22}-H_{12}^2=1   
\end{equation}
\par
Indeed, we know this famous fact that the subspace of $PW_x$, $B_{\tau(x)}:=\{F\in B_x=B(E_x)=PW_x:\tau (F)\leq \tau_x (E)\}$, is also a de Branges space and $B_{\tau(x)}=B(E_a)=B_a=PW_a$, where $a=\max\{t\in \mathbb{R}:\tau_t(E)\leq \tau_x (E)\}$, then $a\leq x$ since $PW_a$ is a subspace of $PW_x$. On the other hand, recall that the exponential type formula is given by 
\[\tau_x (E)=\int_0^x\sqrt{\det H(s)}ds\]
This implies $\sqrt{\det H(s)}\neq0$ a.e. on the interval $[x,x_0]$, otherwise $a=x_0>x$. Since the choice of $x$ is arbitrary, we conclude that $\sqrt{\det H(s)}\neq0$ a.e. on $[0,N]$; as a consequence, $a=x$, namely $B_{\tau(x)}=B_x=PW_x$. The Paley-Wiener theorem states that the exponential type of a function from $PW_x$ is at most $x$, which implies that $\tau_x (E)\geq x$. Assume $\tau_x (E)> x$, then there must be a number $x_1<x$ such that $\tau_{x_1} (E)=\int_0^{x_1}\sqrt{\det H(s)}ds =x$, then the space $\{F\in B_x=B(E_x)=PW_x:\tau (F)\leq \tau_{x_1} (E)=x\}=PW_x=PW_a$, where $a=\max\{t\in \mathbb{R}:\tau_t(E)\leq \tau_{x_1} (E)=x\}=x_1$, which is a contradiction. It follows from $\tau_x (E)=\int_0^x\sqrt{\det H(s)}ds= x$ that $\det H=1$.
\par
Equations (16)-(18) give us a system (a.e.). To solve it, we first observe that $j(x,-t)=\overline{j}(x,t)$ and $k(x,-t)=\overline{k}(x,t)$ due to the uniqueness of the solution of those corresponding integral equations.  If $A(x_0)=0$ for some $x_0\in (0,N]$, i.e., $j(x_0,-x_0)=\overline{j}(x_0,x_0)=0$, then it follows from the first equation given at the beginning of the proof that 
$j_t(x_0,t)=0$ (the uniqueness again); however, this implies that $j(x_0,t)=0$, which cannot be the solution of the original integral equation, hence we conclude that $A\neq 0$ on $[0,N]$ (notice that $j(0,0)=\frac{1}{2}$).
\par
Now, we claim that $D\neq 0$ on $(0,N]$. Indeed, if this is not true, then we have $j(x_0,-x_0)k(x_0,x_0)=j(x_0,x_0)k(x_0,-x_0)$ for some $x_0\in (0,N]$. The same calculation as above implies that 
\[k(x_0,-x_0)j_t(x_0,t)-j(x_0,-x_0)(k_c)_t(x_0,t)=0\]
If we integrate on both sides with respect to $t$ from $-x$ to $x$ and recall $k_c(x_0,-x_0)=k(x_0,-x_0)$, $k_c(x_0,x_0)=k(x_0,x_0)-i$, then we conclude that $j(x_0,-x_0)=0$, which is in contradiction to $A\neq 0$. The limit of $D$ is from the continuity of $j$ and $k_c$.
\par
It is easy to solve the system now, Since $A,B,C$ are continuous except $x=0$, it follows that $H$ can be chosen so that it (the representative) is continuous on $(0,N]$ and $H(0)=I$ as the right limit; moreover, corollary \ref{c2.2} gives the existence of a continuous Dirac operator.
\end{proof}
A significant application of Theorem \ref{t4.3} is the case $\phi \in L^2$.
\begin{theorem}
\label{t4.4}
Assume $\phi \in \Phi_N\cap L^2(-2N,2N)$, then the corresponding canonical system is a Dirac system with an absolutely continuous measure, i.e., a classical Dirac system with an $L^1(0,N)$ function as its coefficient.     
\end{theorem}
\begin{proof}
Thanks to  Corollary \ref{cor2.1} and Theorem \ref{t4.3}, it is enough to show  $j(x,x)$ and $k(x,x)$ are absolutely continuous.
\par
Theorem \ref{t4.2} and Corollary \ref{c4.3} combined says that 
\[(p_x(x,t)+p_t(x,t))+\int_{-x}^x\phi(t-s)(p_x(x,s)+p_t(x,s))=-2\phi(t+x)p(x,-x)\]
Let's denote $w(x,t):=p_x(x,t)+p_t(x,t)$. We want to investigate $w(x,x)$, which should be the derivative of $p(x,x)$; however, there is one issue we need to take into consideration first: we don't know if $w(x,x)$ is measurable, and it is possible that $w(x,x)$ is not defined almost everywhere.
\par
Let's consider $\{\phi_n\}\subset C_c^\infty(-2N,2N)\cap \Phi_N$ such that $\phi_n\to \phi$ in $L^2(-2N,2N)$ (the same as Corollary \ref{c4.1}) and equations
\[w_n(x,t)+\int_{-x}^x\phi_n(t-s)w_n(x,s)ds=-2\phi_n(t+x)p_n(x,-x)\]
where $p_n(x,t)$ denotes $j_n(x,t)$ or $(k_c)_n(x,t)$ and $w_n(x,s)=(p_n)_t(x,t)+(p_n)_x(x,t)$.
\par
It's clear that $w_n(x,t) \in C(\overline{\Delta_N})$, hence it follows that the integral \[I_n(x):=\int_{-x}^x\phi_n(x-s)w_n(x,s)ds=-w_n(x,x)-2\phi_n(2x)p_n(x,-x)\]
is well-defined everywhere and continuous on $[0,N]$. Moreover, by applying the routine(see Section 4.5), we have for any $x$ that $w_n\to w$ in $L^2$.
\par
For any fixed $x\in[0,N]$, Holder' inequality shows that $I(x):=\int_{-x}^x\phi(x-s)w(x,s)ds$ is also well-defined, and  $\lim\limits_{n\to \infty}I_n(x)=I(x)$. As a consequence, $I(x)$ is measurable as the pointwise limit of measurable functions, hence we conclude that $w(x,x)=-I(x)-2\phi(2x)p(x,-x)$ is measurable.
\par
Next, let's evaluate $\int_0^xI(t)dt$. Since we have
\[\int_0^N|I(t)|dt\leq\int_0^N\int_{-t}^t|\phi(t-s)|\cdot|w(t,s)|dsdt\]
By Holder's inequality, the following estimate holds
\[\int_0^N|I(t)|dt\leq\lVert\phi\rVert_{L^2}\int_0^N\lVert w(t,\cdot)\rVert_{L^2(-t,t)}dt\]
The solution $w(x,t)$ can be written as 
\[w(x,t)=-2p(x,-x)(1+T_{\phi}^x)^{-1}(\phi(t+x))\]
hence
\[\lVert w(x,\cdot)\rVert_{L^2(-x,x)}\leq Const.\lVert(1+T_{\phi}^x)^{-1}\rVert\cdot\lVert \phi\rVert_{L^2}\]
We have already proved that the right-hand side is uniformly bounded (in the proof of Theorem \ref{t4.1}), hence it follows that $I(x)\in L^1[0,N]$, and therefore, $w(x,x)\in L^1[0,N]$.
\par
We have 
\[\int_0^kw_n(x,x)dx=-\int_0^k\int_{-x}^x\phi_n(x-s)w_n(x,s)dsdx-2\int_0^k\phi_n(2x)p_n(x,-x)dx\]
Obviously, the Fubini theorem works here, so we have
\[\int_0^kw_n(x,x)dx=-\int_{-k}^k\phi_n(k-s)p_n(k,s)ds\]
the right-hand side equals $p_n(k,k)-g_n(k,k)$, and $g_n(k,k)=g(k,k)$ is just a constant, i.e., $p_n(k,k)=g(k,k)+\int_0^kw_n(x,x)dx$ is absolutely continuous. If we apply the dominated convergence theorem to $\int I(s)ds$, we conclude that $\lim\limits_{n\to\infty}\int_0^kw_n(x,x)dx=\int_0^kw(x,x)dx$; moreover, Since $p_n(x,t)\to p(x,t)$ in the sense of supreme norm, if we take limits on both sides, we finally get
\[p(x,x)=g(x,x)+\int_0^kw(x,x)dx\]
that is, $p(x,x)$ is absolutely continuous.
\end{proof}
\subsection{The Proofs of Theorem \ref{t4.2} and Corollary \ref{c4.3}}
\subsubsection{The Proof of Theorem \ref{t4.2}}
It is clear that $p(x,t)\in C[-x,x]$ for a fixed $x\in (0,N]$ since $1+C_{\phi}^x$ is a bijection. We want to show the following lemma:
\begin{lemma}
 $p(x,t)$ is continuous at $(0,0)$.  
\end{lemma}
\begin{proof}
 We pick up $0<\epsilon<1$.  Recall that $\lVert C_\phi^xf\rVert\leq \lVert f\rVert_{C[-x,x]}\sup\limits_{t\in [-x,x]}\int_{t-x}^{t+x}|\phi(s)|ds$, if we pick up a small $N_1$ such that if  $b-a\leq 2N_1$, then $\int_a^b|\phi(s)|ds<\epsilon$ (we can always do this since the measure $\int_A|\phi(s)|ds$ is absolutely continuous), then we have $\lVert C_\phi^x\rVert<\epsilon$ for all $x\in (0,N_1]$. It is clear that $(1+C_\phi^x)^{-1}=\sum\limits_{n=0}^\infty(-C_\phi^x)^n$ and $\lVert(1+C_\phi^x)^{-1}\rVert\leq \frac{1}{1-\epsilon}$ for all $x\in (0,N_1]$. Consider a  small triangle $\overline{\Delta_{N_1}}$, then it follows from $p(0,0)=g(0,0)$ that 
 \[p(x,t)-p(0,0)=(1+C_\phi^x)^{-1}(g(x,t)-g(0,0)-C_\phi^xg(0,0))\]
 Since $g(x,t) \in C(\overline{\Delta_N})$, then $N_1$ can be chosen small enough so that $|g(x,t)-g(0,0)|<\epsilon$, hence the equality above gives
\[|p(x,t)-p(0,0)|\leq 
\frac{\epsilon(1+|g(0,0)|)}{1-\epsilon}\]
This implies that $p(x,t)$ is continuous at $(0,0)$.
\end{proof}
The rest is the continuity on $\Theta_{N_1}:=\overline{\Delta_N}\setminus \overline{\Delta_{N_1}}\cup \{(N_1,t):|t|\leq N_1\}$ for any $N_1>0$. To do this, we introduce a method to transfer equation (14) into a nicer one, see Chapter 6 in \cite{3} for instance.
\par
Let's fix a $N_1$, and denote by $h$ a continuous function as follows:
\[h:\Theta_{N_1} \to [N_1,N]\times[-1,1], h(x,t)=(x,\frac{t}{x})\]
Notice that the construction of the function makes sense because we've already avoided $x=0$; moreover, on $\Theta_{N_1}$, $p$ solves the equation (14)
\[p(x,t)+\int_{-x}^x\phi (t-s)p(x,s)ds=g(x,t)\]
if and only if on $[N_1,N]\times[-1,1]$, $q:=p\circ h^{-1}$ solves
\begin{equation}
 q(x,t)+\int_{-1}^1x\phi (xt-xs)q(x,s)ds=g\circ h^{-1}(x,t)   
\end{equation}
We denote the integral operator derived from  (19) from $C[-1,1]$ to $C[-1,1]$ for a fixed $x\in [N_1,N]$ by $1+K_\phi^x$, and with the help of $C_\phi^x$, we know that $1+K_\phi^x$ is boundedly invertible.
\begin{lemma}
\label{l4.11}
The map $\kappa_{\phi}(x):=K_{\phi}^x:(0,N]\to B(C[-1,1])$ is continuous.
\end{lemma}
\begin{proof}
 Assume $0<\epsilon<1$ and fix $x_0\in (0,N]$, we pick $x$ such that $|x-x_0|<\delta$ for some proper $\delta>0$, and we want to figure out how proper this $\delta$ should be. 
 \par
 For $f\in C[-1,1]$, we have
 \[\lVert(K_\phi^x- K_\phi^{x_0})f\rVert\leq \rVert f\rVert\sup\limits_{t\in [-1,1]}\int_{-1}^1|x\phi (xt-xs)-x_0\phi (x_0t-x_0s)|ds\]
 As before, we can pick $\{\phi_n\}\subset C_c^\infty(-2N,2N)\cap \Phi_N$ such that $\phi_n\to \phi$ in $L^1(-2N,2N)$. Let's pick one $\phi_n$ such that $\lVert\phi-\phi_n\rVert<\frac{\epsilon}{8\lVert(1+K_\phi^{x_0})^{-1}\rVert}$, then it follows that 
 \begin{align*}
 &\int_{-1}^1|x\phi (xt-xs)-x_0\phi (x_0t-x_0s)|ds&\\
 &\leq \int_{-1}^1|x\phi (xt-xs)-x\phi_n (xt-xs)|ds+\int_{-1}^1|x_0\phi_n (x_0t-x_0s)-x_0\phi (x_0t-x_0s)|ds&\\
 &+\int_{-1}^1|x\phi_n (xt-xs)-x_0\phi_n (x_0t-x_0s)|ds&
 \end{align*}
The first and the second terms are not greater than $\lVert\phi-\phi_n\rVert$; moreover, we have 
\begin{align*}
 &\int_{-1}^1|x\phi_n (xt-xs)-x_0\phi_n (x_0t-x_0s)|ds&\\ 
 &\leq \int_{-1}^1x|\phi_n (xt-xs)-\phi_n (x_0t-x_0s)|ds+|x-x_0|\int_{-1}^1|\phi_n (x_0t-x_0s)|ds&
\end{align*}
Since $\phi_n$ is uniformly continuous, if $\delta$ is chosen properly small enough such that $|\phi_n (xt-xs)-\phi_n (x_0t-x_0s)|\leq \frac{\epsilon}{8N||(1+K_\phi^{x_0})^{-1}||}$, and notice the second term can be also small, then an easy calculation will give us a proper $\delta$ such that 
\[\lVert K_\phi^x- K_\phi^{x_0}\rVert<\frac{\epsilon}{\lVert(1+K_\phi^{x_0})^{-1}\rVert}\]
\end{proof}
\begin{lemma}
 $p(x,t)$ is continuous on $\Theta_{N_1}$ for any $N_1>0$.
\end{lemma}
\begin{proof}
As before, we assume $0<\epsilon<\frac{1}{2}$ and fix $x_0\in [N_1,N]$, and pick $x$ such that $|x-x_0|<\delta$ for some proper $\delta>0$ such that 
\[\lVert K_\phi^x- K_\phi^{x_0}\rVert<\frac{\epsilon}{\lVert(1+K_\phi^{x_0})^{-1}\rVert}\]
Consider the following equation about $K_\phi^x$:
\[(1+K_\phi^{x_0})^{-1}(g\circ h^{-1}(x,t))=q(x,t)+(1+K_\phi^{x_0})^{-1}(K_\phi^x- K_\phi^{x_0})(q(x,t))\] 
Since we have $\sup\limits_{\overline{\Delta_N}}|g(x,t)|\leq M$ for some $M$, then 
\[\lVert q(x,\cdot)\rVert_{C[-1,1]}\leq M\rVert(1+K_\phi^{x_0})^{-1}\rVert+\epsilon\lVert q(x,\cdot)\rVert_{C[-1,1]}\]
equivalently,
\[\sup\limits_{|x-x_0|<\delta}\lVert q(x,\cdot)\rVert_{C[-1,1]}\leq \frac{M\rVert(1+K_\phi^{x_0})^{-1}\rVert}{1-\epsilon}\]
We also have 
\[(1+K_\phi^{x_0})^{-1}(g\circ h^{-1}(x_0,t))=q(x_0,t)\]
If we modify $\delta$ a little bit, then we can even have 
\[\lVert g\circ h^{-1}(x,\cdot)-g\circ h^{-1}(x_0,\cdot)\rVert_{C[-1,1]}<\epsilon\]
Now, it is from 
\[q(x_0,t)-q(x,t)=(1+K_\phi^{x_0})^{-1}((K_\phi^{x}-K_\phi^{x_0})q(x,t)+g\circ h^{-1}(x_0,t)-g\circ h^{-1}(x,t))\]
that
\[\lVert q(x_0,\cdot)-q(x,\cdot)\rVert_{C[-1,1]}\leq (\lVert(1+K_\phi^{x_0})^{-1}\rVert+2M\lVert(1+K_\phi^{x_0})^{-1}\rVert)\epsilon\]
If we consider $x$ and $t$ simultaneously, we have 
\[|q(x,t)-q(x_0,t_0)|\leq |q(x,t)-q(x_0,t)|+|q(x_0,t)-q(x_0,t_0)|\]
This implies $q(x,t)\in C([N_1,N]\times[-1,1])$; moreover, as the composition of two continuous functions, $p(x,t)\in C(\Theta_{N_1})$.
\end{proof}
It is straightforward now to get the conclusion below:
\begin{corollary}
$p(x,t) \in C(\overline{\Delta_N})$.    
\end{corollary}
Now, it's time to prove the second part. Fix $x\in (0,N]$, and  pick $\{\phi_n\}\subset C_c^\infty(-2N,2N)\cap\Phi_N$ and $\phi_n\to \phi$ in $L^1(-2N,2N)$.We have solutions $p_n$ satisfying \[(1+C_{\phi_n}^x)p_n=g(x,t)\]
or equivalently,
\[p_n(x,t)=-\int_{-x}^x\phi_n(t-s)p_n(x,s)ds+g(x,t)\]
Since $\phi_n$ is differentiable and its derivative is bounded, then by applying the mean value theorem and dominated convergence theorem to $\phi_n$, we conclude that $\int_{-x}^x\phi_n(t-s)p_n(x,s)ds$ is differentiable as a function of $t$, and the derivative is $\int_{-x}^x\phi_n'(t-s)p_n(x,s)ds$. As $p_n(x,t) \in C(\overline{\Delta_N})$ is bounded, and $\phi_n'$ is also bounded, then we conclude that by mean value theorem again, $\int_{-x}^x\phi_n(t-s)p_n(x,s)ds$ is Lipchitz, hence absolutely continuous for $t$, thus it follows that $p_n$ is absolutely continuous. We denote those derivatives by $(p_n)_t$, and we have
\[(1+L_{\phi_n}^x)(p_n)_t=g_t(x,t)+\phi_n(t-x)p_n(x,x)-\phi_n(t+x)p_n(x,-x)\]
Recall that 
\[(1+C_{\phi_n}^x)^{-1}=(1-(1+C_{\phi}^x)^{-1}C_{\phi-\phi_n}^x)^{-1}(1+C_{\phi}^x)^{-1}\]
and for small $0<\epsilon<\frac{1}{2}$, we have for all large $n$
\[\lVert C_{\phi-\phi_n}^x\rVert\leq \lVert\phi-\phi_n\rVert_{L^1}\leq \frac{\epsilon}{\lVert(1+C_{\phi}^x)^{-1}\rVert}\]
Therefore, it follows that 
\[\lVert(1+C_{\phi_n}^x)^{-1}-(1+C_{\phi}^x)^{-1}\rVert\leq 
\frac{\lVert(1+C_{\phi}^x)^{-1}\rVert}{1-\epsilon}\epsilon<2\lVert(1+C_{\phi}^x)^{-1}\rVert\epsilon\]
For $1+L_{\phi_n}^x$, we have an analogous conclusion with the help of Young's inequality.
\par
From $p_n=(1+C_{\phi_n}^x)^{-1}g$ we conclude that $\{p_n\}$ is a Cauchy sequence in $C[-x,x]$ that converges to $p(x,t)$, hence $\{p_n\}$ is uniformly bounded. As a consequence, $g_t(x,t)+\phi_n(t-x)p_n(x,x)-\phi_n(t+x)p_n(x,-x)$ is convergent to $g_t(x,t)+\phi(t-x)p(x,x)-\phi(t+x)p(x,-x)$ in $L^1[-x,x]$. It follows from an analogous argument (but a little more complicated) that $\{(p_n)_t\}$ is Cauchy in $L^1[-x,x]$, i.e., there is a function in $L^1[-x,x]$, denoted by $p_t$, such that $\lim\limits_{n\to\infty}(p_n)_t\to p_t$. If we take limits for $(1+L_{\phi_n}^x)(p_n)_t$, we get the equation we need; moreover, we have
\[p_n(x,t)=p_n(x,0)+\int_0^t(p_n)_t(x,s)ds\]
Once we take limits on both sides, we get
\[p(x,t)=p(x,0)+\int_0^tp_t(x,s)ds\]
\subsubsection{The Proof of Corollary \ref{c4.3}}
We first discuss the uniform upper bound of the following operators.
\begin{lemma}
 \[\sup\limits_{x\in(0,N]}\{\lVert (1+T^x_\phi)^{-1}\rVert\}\leq t(\phi)\]
 \[\sup\limits_{x\in(0,N]}\{\lVert (1+C^x_\phi)^{-1}\rVert\}\leq c(\phi)\] 
 \[\sup\limits_{x\in(0,N]}\{\lVert (1+L^x_\phi)^{-1}\rVert\}\leq l(\phi)\]
 
 where $t(\phi),l(\phi),c(\phi)$ are constants determined only by $\phi$.
\end{lemma}
\begin{proof}
The first inequality has been proved in Theorem \ref{t4.1}.
\par
To prove the second one, we define on $(0,N]$ a function for $\phi\in \Phi_N$ as follows:
\[\gamma_{C_\phi}(x):=\inf \{\lVert(1+C_\phi^x)f\rVert: f\in C[-x,x], \lVert f\rVert\geq 1\}\]
We claim that $\gamma_{C_\phi}(x)$ is continuous on $(0,N]$. Indeed, with the help of $K_{\phi}^x$, we have
\[\gamma_{C_\phi}(x)=\inf \{\lVert(1+K_\phi^x)q\rVert: q=f\circ h^{-1}\in C[-1,1], \lVert q\rVert\geq 1\}\]
Because
\[\lVert(1+K_\phi^x)q\rVert\leq \lVert(1+K_\phi^{x_0})q\rVert+\lVert K_\phi^x-K_\phi^{x_0}\rVert\cdot\lVert q\rVert_{C[-1,1]}\]
Then it follows that 
\[\gamma_{C_\phi}(x)\leq \gamma_{C_\phi}(x_0)+\lVert K_\phi^x-K_\phi^{x_0}\rVert\]
If we switch $K_\phi^x$ and $K_\phi^{x_0}$, then we have 
\[|\gamma_{C_\phi}(x_0)-\gamma_{C_\phi}(x)|\leq \lVert K_\phi^x-K_\phi^{x_0}\rVert\]
Recall Lemma \ref{l4.11}, we conclude that $\gamma_{C_\phi}(x)$ is continuous on $(0,N]$. Observe that $\lim\limits_{x\to0^+}\kappa_{\phi}(x)=0$, we have $\lim\limits_{x\to0^+}\gamma_{C_\phi}(x)=1$. if $\inf(\gamma_{C_\phi}(x))=0$, then there must be $x_0\in (0,N]$ such that $\gamma_{C_\phi}(x_0)=0$, but this is impossible due to Theorem \ref{t4.1}, Lemma \ref{l4.7} and the closed graph theorem. So we have the desired inequality.
\par
The third one is from the same argument. We define (the target operator is in $L^1[-1,1]$ rather than $C[-1,1]$, but we still use $K_{\phi}^x$)
\[\gamma_{L_\phi}(x):=\inf \{\lVert(1+L_\phi^x)f\rVert: f\in L^1[-x,x], \lVert f\rVert\geq 1\}\]
and rewrite it as 
\[\gamma_{L_\phi}(x)=\inf \{x\cdot\lVert(1+K_\phi^x)q\rVert: q=f\circ h^{-1}\in L^1[-1,1], x\cdot\lVert q\rVert\geq 1\}\]
$\lVert K_\phi^x-K_\phi^{x_0}\rVert$ now is from Young's inequality.
\end{proof}
In the sequel, we focus on the case $g=\frac{1}{2}$; the other scenario can be handled in the same way (the only nuance is that we need to consider $g_n$ obtained by substituting $\phi$ with $\phi_n$ in $g$).
\par
Recall in the proof of Theorem \ref{t4.2}, we have proved the following inequality: 
\[\lVert(1+C_{\phi_n}^x)^{-1}-(1+C_{\phi}^x)^{-1}\rVert\leq2\lVert(1+C_{\phi}^x)^{-1}\rVert\epsilon\]
Moreover, we can say more about it for all large $n$:
\[\lVert p_n-p\rVert_{C[-x,x]}\leq\frac{1}{2}\lVert(1+C_{\phi_n}^x)^{-1}-(1+C_{\phi}^x)^{-1}\rVert\leq \frac{\epsilon}{c(\phi)}\]
That is, $p_n$ converges to $p$ uniformly not only about $t$ but also about $x$; and as a result, we conclude that
\[\sup\limits_{(x,t)\in \overline{\Delta_N}} |p_n(x,t)|<M\]
for all $n$ and a large number $M$.
\par
We denote the difference quotient of a given function $f$ by
\[S_\epsilon f(x,t)=\frac{f(x+\epsilon,t)-f(x,t)}{\epsilon}\]
A calculation shows that for $\epsilon>0$,
\[(1+C_{\phi_n}^x)S_\epsilon p_n=\frac{1}{\epsilon}(-\int_x^{x+\epsilon}\phi_n(t-s)p_n(x+\epsilon,s)-\int_{-(x+\epsilon)}^{-x}\phi_n(t-s)p_n(x+\epsilon,s))\]
and for $\epsilon<0$,
\[(1+C_{\phi_n}^{x+\epsilon})S_\epsilon p_n=\frac{1}{\epsilon}(-\int_x^{x+\epsilon}\phi_n(t-s)p_n(x,s)-\int_{-(x+\epsilon)}^{-x}\phi_n(t-s)p_n(x,s))\]
\par
Let's deal with the first equation. Since the right-hand side converges to $-\phi_n(t-x)p_n(x,x)-\phi_n(t+x)p_n(x,-x)$ in $C[-x,x]$, hence $S_\epsilon p_n\to (p_n)_x$ in $C[-x,x]$ when $\epsilon\to 0$ from the right-hand side.
\par
To deal with the second equation, consider $(1+K^{x+\epsilon}_{\phi_n})(S_\epsilon p_n\circ h^{-1})$ and $(1+K^{x}_{\phi_n})((p_n)_x\circ h^{-1})$, moreover, we still have \[(1+K_{\phi_n}^{x+\epsilon})^{-1}=(1-(1+K_{\phi_n}^x)^{-1}(K_{\phi_n}^x-K_{\phi_n}^{x+\epsilon}))^{-1}(1+K_{\phi_n}^x)^{-1}\]
it follows that $S_\epsilon p_n\to (p_n)_x$ in $C[-x,x]$ when $\epsilon\to 0$ from the left-hand side.
\par
In all, $p_n$ is differentiable everywhere in term of $x$, and the (partial) derivative is $(p_n)_x$, and 
\[p_n(x,t)=p_n(|t|,t)+\int_{|t|}^x(p_n)_x(s,t)ds\]
Moreover, since $\{p_n\}$ is uniformly bounded in $\overline{\Delta_N}$, it follows that $-\phi_n(t-x)p_n(x,x)-\phi_n(t+x)p_n(x,-x)\to -\phi(t-x)p(x,x)-\phi(t+x)p(x,-x)$ in $L^1[-x,x]$ uniformly in $x\in(0,N]$, then it follows from the uniform bound of $\lVert(1+L_{\phi}^x)^{-1}\rVert$ in term of $x$ that  \[\lim\limits_{n\to\infty}\lVert(p_n)_x-p_x\rVert_{L^1[-x,x]}=0\] 
uniformly in $x$.
\par
This limit implies that $\lim\limits_{n\to\infty}\int_{-x}^x(p_n)_x(x,s)ds=\int_{-x}^x(p)_x(x,s)ds$ uniformly, and moreover, $\int_{-x}^x(p_n)_x(x,s)ds$ are continuous, hence $\int_{-x}^x(p)_x(x,s)ds$ is also continuous. It is easy to see that $\chi_{[-x,x]}(s)\cdot|(p_n)_x|(x,s)$ is a measurable function on $[-N,N]\times\mathbb{R}$, therefore, $\int_{-x}^x|(p_n)_x|(x,s)ds$ is measurable about $x$, hence, as the point limit of $\int_{-x}^x|(p_n)_x|(x,s)ds$,  $\int_{-x}^x|p_x|(x,s)ds$ is also measurable.
If we pick up a large $n$, then we have 
\begin{align*}
    &\int_0^xdm\int_{-m}^m|p_x(m,s)|ds&\\
    &\leq\int_0^x\lVert(p_n)_x-p_x\rVert_{L^1[-m,m]}dm+\int_0^xdm\int_{-m}^m|(p_n)_x(m,s)|ds&
\end{align*}
The first term of the second line is small enough, and $(p_n)_x$ is bounded, hence it follows that $\int_{-x}^x|p_x(x,s)|ds$ is integrable in terms of $x$.
\par
We also have
\begin{align*}
0&=\lim\limits_{n\to\infty}\int_0^xdm\int_{-m}^mh(s)((p_n)_x(m,s)-p_x(m,s))ds\\
  &=\lim\limits_{n\to\infty}\int_{-x}^xdsh(s)\int_{|s|}^x(p_n)_x(m,s)dm-\int_0^xdm\int_{-m}^mh(s)p_x(m,s)ds\\
  &=\lim\limits_{n\to\infty}\int_{-x}^xh(s)(p_n(x,s)-p_n(|s|,s))ds-\int_0^xdm\int_{-m}^mh(s)p_x(m,s)ds
\end{align*}
This gives 
\[\int_{-x}^xh(s)p(x,s)ds=\int_{-x}^xh(s)p(|s|,s)ds+\int_0^xdm\int_{-m}^mh(s)p_x(m,s)ds\]
\vspace{2cm}
\section{Gelfand-Levitan Condition for the Half-line Problem}
In this section, we discuss the half-line problem, that is, a Dirac operator (1) on $L^2[0,\infty)$. This topic is standard, and a related treatment can be found in \cite{8}.
\par
As before, we focus on canonical systems, more precisely, the Weyl function of a canonical system. Theorem \ref{2.2}, Corollary \ref{2.1}, and Corollary \ref{c2.2} imply that the spectral measure $\rho$ from a Dirac operator is the same as that from the corresponding canonical system; moreover, the unitary map $U$ implies $L^2(\mathbb{R}, \rho)$ is isometrically equivalent to $PW_N$ if we consider an interval. We will utilize this idea to establish our theorem.
\begin{definition}
 A function $\phi\in L^1_{loc}(\mathbb{R})$ is called Gelfand-Levitan (G-L), if $\phi$ is in the Gelfand-Levitan set $GL$, i.e., $\phi\in GL:=\bigcap \limits_{N>0}\Phi_N$, and this is interpreted as the restriction of $\phi$ on $(-2N,2N)$, denoted by $\phi_N$, belongs to $\Phi_N$ for any arbitrary $N>0$.
\end{definition}
Let's assume $M>N>0$ and $\phi\in GL$. $\phi_M$ gives us a canonical system $H_M$ on $(0,M)$ with $B(E_M)=PW_M$. Recall the construction of $H_M$, it follows that $B(E_N)=PW_N\subset PW_M$, and the norm is evaluated by considering the restriction of $\phi_M$ on $(-2N,2N)$, which is indeed $\phi_N$, and the canonical system $H_N$ on $(0,N)$ given by $\phi_N$ is the restriction of $H_M$ on $(0,N)$.
\par
We also recall that the Weyl function of a canonical system, as a Herglotz function, has the following representation:
\[m(z)=a+bz+\int_\mathbb{R}(\frac{1}{t-z}-\frac{t}{t^2+1})d\rho(t)
=a+\int_{\mathbb{R}_\infty} \frac{1+tz}{t-z}dv(t)\]
where $a\in\mathbb{R}$, $b\geq 0$, and $\rho$ a positive Borel measure on $\mathbb{R}$ with $\int_\mathbb{R}\frac{d\rho(t)}{t^2+1}<\infty$ which is indeed the spectral measure; moreover, $dv(t)=\frac{1}{1+t^2}d\rho(t)+b\delta_\infty$.
\begin{definition}
let $\rho$ be a positive Borel measure on $\mathbb{R}$ with  $\int_\mathbb{R}\frac{d\rho(t)}{t^2+1}<\infty$. Define a signed measure $d\sigma(t)=d\rho(t)-\frac{1}{\pi}dt$. We say $\rho$ satisfies the G-L condition if $\widehat{\sigma}=2\phi\in GL$ in the sense of distributions.  
\end{definition}
\begin{theorem}
\label{t5.1}
We denote a Dirac operator with the boundary condition $f_2(0)=0$ on $[0,\infty)$ by $Jf'-\mu f$.
\begin{enumerate}
    \item For a Dirac operator  with $\mu$ an $L^1_{loc}(0,\infty)$ function, its spectral measure $\rho$ satisfies the G-L condition.
    \item For a measure $\rho$ satisfying the G-L condition, it is the spectral measure of a Dirac operator with $\mu$ a continuous measure.
    \item For a measure $\rho$ satisfying the G-L condition and the corresponding G-L function $\phi\in L^2_{loc}(0,\infty)$, it is the spectral measure of a Dirac operator with $\mu$ an $L^1_{loc}(0,\infty)$ function.
\end{enumerate}
\end{theorem}
\begin{proof}
Assume we have such a Dirac operator on $[0,\infty)$ with the spectral measure $\rho$. The restriction of this operator on $[0,N]$ gives us a Weyl function $m_N$ by giving it a boundary condition at $N$  (for instance, $\beta=\frac{\pi}{2}$).  Based on Weyl theory, those functions $m_N$ converge to the Weyl function of the half line $m$ in the sense of locally uniform convergence, and spectral measures $\rho_N\to \rho$ weakly for every function $(t^2+1)f(t)\in C_0(\mathbb{R})$.
\par
On $[0,N]$, Theorem \ref{3.4} gives $\phi_N\in \Phi_N$ and for $F\in PW_N$
\[\lVert F\rVert^2_{B(E_N)}=2\langle f,f+\phi_N{*}f\rangle\]
Those $\phi_N$, of course, give a function $\phi\in GL$, and we can replace $\phi$ for $\phi_N$ above. Recall the unitary map from the spectral representation and the definition of $Uf(z)$, we know $L^2(\mathbb{R},\rho_N)=B(E_N)$ isometrically. A calculation (recall our definition is $\widehat{f}(z)=\int f(t)e^{itz}dt$) gives for all $f\in L^2(-N,N)$ with $d\sigma_N(t)=d\rho_N(t)-\frac{1}{\pi}dt$ that
\[\langle f,2\phi{*}f\rangle=\int_{\mathbb{R}}\lvert F\rvert^2d\sigma_N\]
As the Fourier transform of a tempered distribution, $\widehat{\sigma_N}$ is again tempered defined as $(\widehat{\sigma_N},g)=(\sigma_N,\widehat{g})$.
\par
For all $f\in C_0^{\infty}(-N,N)$ and write $f_r(x)=\overline{f}(-x)$, we have $\widehat{g}=|\widehat{f}|^2$ if $g:=f{*}f_r$, therefore, if we consider $f\in C_0^{\infty}(-N,N)$, it follows that
\[\langle f,2\phi{*}f\rangle=2\int_{-N}^Ndx\overline{f}(x)\int_{-2N}^{2N}\phi(t)f(x-t)dt=\int_{-2N}^{2N}2\overline{\phi}(t)g(t)dt\]
Thus for all $g=f{*}f_r$ with $f\in C_0^{\infty}(-N,N)$, we have
\begin{equation}
(\widehat{\sigma_N},g)=(2\overline{\phi},g)    
\end{equation}
Now, we want to show that (20) is also true for an arbitrary $g\in C_0^{\infty}(-2N,2N)$. The argument is (literally) the same as that in the proof of Theorem 6.20 in \cite{1}, so we only state it briefly: if this is true for $f{*}f_r$, then it is also true for $h{*}f_r$ by polarization, and so as the linear combinations of such functions. To approximate a general $g \in C_0^{\infty}(-2N,2N)$, we can rewrite $g=g_1+g_2+g_3$ whose supports are of diameter less that $2N$. For such a $g_i$, there must be a constant $a$ such that $g_i(t-a)\in C_0^{\infty}(-N,N)$, and this function $g_i(t-a)$ can be approximated by taking convolutions with an approximate identity.
\par
For a test function $g \in \mathcal{D}(\mathbb{R})$, there is a constant $N$ such that $g\in C_0^{\infty}(-2N,2N)$, then we must have
\[(\widehat{\sigma_N},g)=(2\overline{\phi},g)\]
Since $\widehat{g}\in \mathcal{S}(\mathbb{R})$, hence we have
\[\lim\limits_{N\to\infty}\int_{\mathbb{R}}\widehat{g}d\sigma_N=\int_{\mathbb{R}}\widehat{g}d\sigma=(\widehat{\sigma},g)\]
The last one is true because, as a spectral measure, $\int_\mathbb{R}\frac{d\rho(t)}{t^2+1}<\infty$ automatically, so $\widehat{\sigma}$ is a tempered distribution.
\par
We have proved $\widehat{\sigma}=2\overline{\phi}$, and $\overline{\phi}\in GL$, namely, $\rho$ satisfies G-L condition.
\par
Conversely, assume we have such a measure $\rho$. Since $\widehat{\sigma}=2\overline{\phi}\in GL$, we can consider $\phi_N\in \phi_N$ as the restriction of $\phi$ on $(-2N,2N)$. Lemma \ref{l4.9} and Theorem \ref{t4.3} (and Theorem \ref{t4.4}) give us a canonical system $H$ on $[0,\infty)$, and the restriction of $H$ on $[0,N]$ gives $\phi_N$. As we have proved, the spectral measure of $H$, denoted by $\rho_{real}$$(\sigma_{real})$, satisfies $\widehat{\sigma}=\widehat{\sigma_{real}}$, or equivalently, $\widehat{\rho}=\widehat{\rho_{real}}$ in the sense of distributions; moreover, in the sense of tempered distributions by extending them continuously.
\par
For $g\in \mathcal{S}(\mathbb{R})$, we have
\[\int_{\mathbb{R}}\widehat{g}d\rho=\int_{\mathbb{R}}\widehat{g}d\rho_{real}\]

Since the Fourier transform is a bijection on $\mathcal{S}(\mathbb{R})$, so actually we don't need Fourier transform here:
\[\int_{\mathbb{R}}gd\rho=\int_{\mathbb{R}}gd\rho_{real}\]
 Let's consider complex measures rather than positive measures, namely,
\[\int_{\mathbb{R}}(t^2+1)g\cdot\frac{1}{t^2+1} d\rho=\int_{\mathbb{R}}(t^2+1)g\cdot \frac{1}{t^2+1}d\rho_{real}\] 
For any $f\in \mathcal{S}(\mathbb{R})$, the equation $(t^2+1)g=f$ will give a solution $g$ in $\mathcal{S}(\mathbb{R})$, hence, we have for $g\in \mathcal{S}(\mathbb{R})$
\[\int_{\mathbb{R}}g\cdot\frac{1}{t^2+1} d\rho=\int_{\mathbb{R}}g\cdot \frac{1}{t^2+1}d\rho_{real}\] 
As two Borel measures, we only need to focus on open intervals; hence, we consider a characteristic function $\chi_{(a,b)}(x)$. It is standard to find a sequence from $\mathcal{S}(\mathbb{R})$ such that $f_n\to \chi_{(a,b)}(x)$ pointwisely and $0\leq f_n\leq1$(for instance, consider $h(x)=\begin{cases}
    e^{-\frac{1}{x}}&\text{if  } x>0\\
    0&\text{if } x\leq 0
\end{cases}$ and $k(x)=\frac{h(x)}{h(x)+h(1-x)}$, then use $k$ to construct this sequence). By applying Lebesgue's dominated convergence theorem, we have
\[\frac{1}{t^2+1} d\rho=\frac{1}{t^2+1}d\rho_{real}\]
namely, $d\rho=d\rho_{real}$ is the spectral measure of the Dirac operator given before.
\end{proof}
In Part 2 and Part 3, the Dirac operator is uniquely determined by the spectral measure. To see this, we recall that for a Dirac operator with the boundary condition, its Weyl function is given by (Claim 5.6 in \cite{12}) 
\[m(z)=a+\int_\mathbb{R}(\frac{1}{t-z}-\frac{t}{t^2+1})d\rho(t)\]
where $\rho(\mathbb{R})=\infty$ and $\rho$ is not compact supported.
\par
If there are two different Dirac operators, then $m_1(z)=c+m_2(z)$ for some constant $c$, therefore we have 
\[H_1=\begin{pmatrix}
    1&0\\
    -c&1
\end{pmatrix}H_2\begin{pmatrix}
    1&-c\\
    0&1
\end{pmatrix}\]
Since $H_1(0)=H_2(0)=I$, we get $c=0$, i.e., we cannot have two different Dirac operators sharing the same spectral measure.

\vspace{1cm}
DEPARTMENT OF MATHEMATICS, 3900 UNIVERSITY BLVD., UNIVERSITY OF TEXAS AT TYLER, TYLER, TX, 75799
\\
Email address: jzeng@uttyler.edu
\end{document}